\documentclass[12pt]{amsart}
\usepackage{latexsym}
\usepackage{amsfonts}
\usepackage{amsmath, amssymb, amsthm,amscd,epsfig}
\usepackage{tikz}
\usepackage{diagrams}
\usepackage{mathrsfs}
\usetikzlibrary{arrows,backgrounds,decorations}

\newcommand{\Z}{\mathbb{Z}}
\newcommand{\R}{\mathbb{R}}
\newcommand{\C}{\mathbb{C}}

\newcommand{\st}{\, \mid \,}

\newcommand{\etale}{$\acute{\textrm{e}}$tale }

\newcommand{\cp}{\mathbb{C}}
\newcommand{\bh}{\mathcal{B}(\mathcal{H})}

\newcommand{\h}{\mathcal{H}}
\newcommand{\ho}{\overline{\mathcal{H}}}

\newcommand{\range}{Range}

\newcommand{\sor}{Source}

\newcommand{\ep}{\varepsilon}

\newcommand{\gs}{G^s(X,\varphi,Q)}
\newcommand{\gu}{G^u(X,\varphi,P)}
\newcommand{\gh}{G^h(X,\varphi)}

\newcommand{\tensor}{\otimes}

\newcommand{\D}{\mathfrak{D}}

\newcommand{\Q}{\mathbb{Q}}

\newcommand{\K}{\mathcal{K}}
\newcommand{\csa}{$C^*$-algebra}
\newcommand{\te}{\otimes}
\newcommand{\x}{\times}
\tikzstyle{vertex}=[circle,draw,fill=gray!20,thick]
\tikzstyle{goto}=[->,shorten >=1pt,>=stealth,semithick]
\DeclareMathOperator{\Ext}{\ensuremath{\mathop{Ext}}}
\DeclareMathOperator{\Hom}{\ensuremath{\mathop{Hom}}}
\DeclareMathOperator{\rank}{\ensuremath{\mathop{rank}}}
\DeclareMathOperator{\image}{\ensuremath{\mathop{image}}}
\DeclareMathOperator{\cok}{\ensuremath{\mathop{cokernel}}}

\newtheorem{defn}{Definition}[section]

\newtheorem{prop}[defn]{Proposition}
\newtheorem{cor}[defn]{Corollary}

\swapnumbers
\theoremstyle{plain}
\newtheorem{theorem}{Theorem}[section]
\newtheorem{lemma}[theorem]{Lemma}

\newtheorem{w-vn}[theorem]{Weyl - von Neumann Theorem}

\theoremstyle{definition}
\newtheorem{definition}[theorem]{Definition}

\theoremstyle{remark}
\newtheorem*{remark}{Remark}

\begin{document}
\title{K-Theoretic Duality for Hyperbolic Dynamical Systems}
\author{Jerome Kaminker}
\address{Department of Mathematical Sciences, IUPUI, Indianapolis,
IN} 
\curraddr{Department of Mathematics, UC Davis, Davis, CA 95616} 
\email{kaminker@math.ucdavis.edu}
\author{Ian F. Putnam$^1$}
\thanks{1. Research supported in part by an NSERC Discovery Grant}
\address{Department of Mathematics and Statistics, University of Victoria, Victoria, B.C., Canada V8W 3R4}
\email{putnam@math.uvic.ca}
\author{Michael F. Whittaker}
\address{Department of Mathematics and Statistics, University of Victoria, Victoria, B.C., Canada V8W 3R4}
\email{mfwhittaker@gmail.com}
\curraddr{Department of Mathematics, Wollongong University, Australia}

\maketitle

\begin{abstract}
The K-theoretic analog of Spanier-Whitehead duality for noncommutative $C^*$-algebras is shown to hold for the Ruelle algebras associated to irreducible Smale spaces.  This had previously been proved only for shifts of finite type.  Implications of this result as well as relations to the Baum-Connes conjecture and other topics are also considered.
\end{abstract}

\section{Introduction}

The goal of this paper is to exhibit a duality between two \csa s  associated to a hyperbolic dynamical system. This is a noncommutative version of Spanier-Whitehead duality from topology.  It turns out that it is a special case of a type of duality which occurs in several different settings.  It will be  described carefully and we will indicate some of the different contexts in which it appears. 

Let us first briefly recall Spanier-Whitehead duality, a generalization of Alexander duality that relates the homology of a subspace of a sphere with the cohomology of its complement. Given a finite complex, $X {\subseteq} S^{n+1}$, consider the map 
\begin{equation}
\Delta : X \times (S^{n+1} \setminus X) \to S^n,                                                                                                                                                                                                                                                                  
                                                                                                                                                                                                                                                                \end{equation}
defined by $\Delta(x,y) = \frac{x-y}{\|x-y\|}$, where the algebraic operations take place in $S^{n+1} \setminus \{ {\text{north pole}}\} \cong \R^{n+1}$. Then one has an isomorphism 
\begin{equation}
 \Delta^* ([S^n]) / :\tilde H_{n-i}(X) \to H^{i} (S^{n+1} \setminus X),
\end{equation}
given by slant product.  This was generalized by Spanier and Whitehead to allow $S^{n+1}\setminus X$ to be replaced by a space of the homotopy type of a finite complex, $\D_n(X)$, for which the analog of the above relations hold.  It is in this form that these notions extend naturally to the noncommutative setting.  

In noncommutative topology, the roles of homology and cohomology are played by K-theory and K-homology.  These have been combined into a bivariant theory by Kasparov, for which one has K-theory, $K_*(A) = KK^*(\C,A)$, and K-homology, $K^*(A) = KK^*(A,\C)$. Kasparov's theory comes equipped with a product that is the analog of the slant product used above.  Of course, noncommutative \csa s play the role of algebras of continuous functions on ordinary spaces.  With this in hand, the generalization of Spanier-Whitehead duality to this setting is easily accomplished. 

Let $A$ be a \csa.  We will consider indices in K-theory to be modulo 2.  A \csa~$\D_n(A)$ is a Spanier-Whitehead n-dual of $A$, (or simply a ``dual`` when the context is clear) if there are duality classes $\delta \in K^n (A \te \D_n(A))$ and $\Delta \in K_n (A \te \D_n(A))$ such that the Kasparov products yields inverse isomorphisms,
\begin{equation}
\label{SWduality}
\begin{aligned}
\delta \te_{\D_n(A)} :K^{i}(\D_n(A)) \to K_{n-i}(A)\\
\te_A \Delta :K_{n-i}(A) \to K^{i}(\D_n(A)).
\end{aligned}
\end{equation}

Note that in the noncommutative case, given $\delta$, it is an additional condition to require the existence of $\Delta $, while in the commutative case this holds automatically \cite{Spa}.

 It is natural to compare this to the noncommutative version of K-theoretic Poincar\'e duality, as introduced by Connes. In general, these notions are different.   The main thing is that Poincar\'e duality relates K-theory and K-homology of the same algebra.  However, this is more restrictive than simply finding a dual algebra, $B$, whose K-theory is isomorphic to the K-homology of $A$.  Even in cases where one can choose an algebra to be its own Spanier-Whitehead dual, care must be taken.  A situation which indicates this is when the algebra is $A = C(S^1)$.  Since $S^3 \setminus S^1$ deformation retracts to $S^1$, we may take $C(S^1)$ itself to be a Spanier-Whitehead 2-dual.  The duality class $\Delta$ lies in $K^0(C(S^1) \te C(S^1))$ and the isomorphism from (\ref{SWduality}) provides isomorphisms $K_*(C(S^1)) \cong K^*(C(S^1))$.  However, $S^1$ is an odd-dimensional $Spin^c$ manifold, so it has a K-theory fundamental class which can be viewed as being in $K^1(C(S^1) \te C(S^1))$ and it provides the usual Poincar\'e duality isomorphisms $K_*(C(S^1)) \cong K^{*+1}(C(S^1))$.  
  
The main result of this paper exhibits this duality for the stable and unstable Ruelle algebras associated to a Smale space, (i.e. a compact space $X$ with a hyperbolic homeomorphism, $\varphi$). The stable and unstable sets for $\varphi$ provide $X$ with the structure of foliated space, in the sense of Moore-Schochet, in two different ways, which are transverse to each other.  The easiest example to visualize is the 2-dimensional torus, $T^2$ with the hyperbolic toral automorphism given by the matrix $ \varphi = [
\begin{smallmatrix}
1 & 1\\
1 & 0                                                                                                                                                                                                                                                                                                                                                                                                                                                                                   \end{smallmatrix}]
$.  The stable and unstable foliated space structures are the Kronecker flows for angles $\theta$ and $\theta^\prime$, respectively, where $\theta = \tan^{-1}(-\gamma)$, $\theta^\prime = \tan^{-1}(\gamma^{-1})$, and $\gamma$ is the golden mean $\gamma=(1+\sqrt{5})/2$.  One may now associate to these structures their Connes foliation algebras, which, in this example, are isomorphic to $A_{\theta} \te \K$ and $A_{\theta^\prime} \te \K$, respectively, where $A_{\theta}$ is the irrational rotation algebra.  These algebras are interesting invariants of the dynamics, and in the present case they happen to be isomorphic, although that is far from true in general.  They are simple algebras with a canonical (semi-finite) trace.   The homeomorphism, $\varphi$, induces an automorphism of each of these algebras, and one can take the associated crossed product algebras, $(A_{\theta} \te \K)\rtimes_{\varphi_*}\Z$ and $(A_{\theta^\prime} \te \K)\rtimes_{\varphi_*}\Z$.  One obtains in this way simple, purely, infinite \csa s.  They are special cases of algebras introduced by the second author in \cite{Put2} and are called the stable and unstable Ruelle algebras associated to a Smale space.  As a consequence of Elliott's classification program, they have the remarkable property of being determined up to isomorphism by their K-theory groups, \cite{Phi}.  It is these algebras which will be shown to be duals. 

We briefly review the general dynamical setting for the duality. The necessary precise definitions will be presented in Section \ref{sec:Smale}. Let $(X,\varphi)$ be a Smale space. In later sections, we will impose some mild dynamical conditions, but they will be suppressed in the introduction.
Consider the groupoids given by the equivalence relations of being stably or unstably equivalent. In general, these are analogs of the holonomy groupoid of a foliation. They are locally compact groupoids which admit Haar systems, so that one may define their $C^\ast$ -algebras, $S(X,\varphi)$ and $U(X,\varphi)$. These are finite direct sums of simple $C^\ast$-algebras which are also separable, nuclear, and stable. They are often known
to be of the type to which Elliott's classification program (in the finite case) applies. They each have a densely defined trace and the map $\varphi$ induces automorphisms, $\alpha_{s}$ , $\alpha_{u}$ , on the algebras, which scale the trace by the logarithm
of the entropy of $\varphi$. As above, we now take the crossed products by these automorphisms to obtain the stable and unstable Ruelle algebras,
\[ R^s(X,\varphi) = S(X, \varphi )   \rtimes_{\alpha_{s}} \Z \quad , \quad R^u(X,\varphi) = U(X, \varphi)   \rtimes_{\alpha_{u}} \Z. \]
  These $C^\ast$-algebras are separable, simple, stable, nuclear, purely infinite, and satisfy the Universal Coefficient Theorem. Thus, according to the purely infinite case of Elliott's program, as developed by Kirchberg and Phillips, they are completely classified by their $K$-theory groups. It is interesting that these algebras, which arose from dynamics and duality theory, turn out to have remarkable properties as $C^\ast$-algebras.

The duality theorem shows that the stable and unstable Ruelle algebras are Spanier-Whitehead duals of each other.  This implies that the $K$-theory of $R^s(X,\varphi) $ is isomorphic, with a dimension shift, to the $K$-homology of $R^u(X, \varphi) $, and vice-versa. 

\begin{theorem} \label{thm:main}
Let $(X,\varphi)$ be an irreducible Smale Space. There exists duality classes $\delta \in KK^1(\C,R^s(X,\varphi) \te R^u(X,\varphi)$ and $\Delta \in KK^1(R^s(X,\varphi)\te R^u(X,\varphi), \C)$ such that
\begin{eqnarray*}
\delta \otimes_{R^u(X,\varphi)} \Delta & \cong & 1_{R^s(X,\varphi)} \quad \textrm{ and } \\
\delta \otimes_{R^s(X,\varphi)} \Delta & \cong & -1_{R^u(X,\varphi)}.
\end{eqnarray*}
\end{theorem}


Taking the Kasparov product with $\delta$ and $\Delta$ yields inverse isomorphisms
\begin{diagram}
K_\ast(R^s(X,\varphi) & \rTo^{\Delta \te_{R^s(X,\varphi)}} & K^{\ast+1}(R^u(X,\varphi))  \\
K^\ast(R^u(X,\varphi)) & \rTo{\delta \te_{R^u(X,\varphi)}} & K_{\ast+1}(R^s(X,\varphi)).
\end{diagram}


We will now describe some of the background for our results. In \cite{Put1}, the second author introduced the algebras studied in the present paper.  They were based on constructions of algebras due to Ruelle in \cite{Rue2}, although the crossed product by the automorphism induced by $\varphi$ was not used there.  Further properties of these algebras were presented in \cite{PS}.  In \cite{KaPu1}, the first two authors worked out a special case of the duality theory when the Smale space was a subshift of finite type. In this case, the Ruelle algebras were known to be isomorphic to stabilized Cuntz-Krieger algebras. Specifically, suppose that  the Smale space was the shift of finite type
associated with a non-negative integral, irreducible matrix, $A$. Denote the dynamical system by $(X_{A} , \varphi_{A} )$. It follows that $R^s(X_{A} , \varphi_{A}) \cong O_{A^{t}} \otimes \mathcal{K}$ and  $R^u(X_{A} , \varphi_{A})\cong O_{A} \otimes \mathcal{K}$.
Following work of D. Evans \cite{Evans} and D. Voiculescu \cite{Voi}, one  considers the full Fock space of a finite dimensional Hilbert space and the associated creation and annihilation operators.  One compresses them to a subspace determined by the matrix $A$ and generates a \csa, $\mathcal{E}$. This    algebra contains the compact operators and one obtains an extension,
\begin{diagram}
0 & \rTo &  \mathcal{K} & \rTo & \mathcal{E} & \rTo & O_{A} \otimes O_{A^{t}} & \rTo & 0.
\end{diagram}
The extension determines an element $\Delta$ in   $K^{1} (O_{A} \otimes O_{A^{t}} )$ and it was shown in \cite{KaPu1} that it induces the required duality isomorphism.  J. Zacharias and I. Popescu \cite{PZ} have extended this type of duality theory to higher rank graph algebras.

There are other sources for examples of duality.  One type is based on an amenable action of a hyperbolic group, $\Gamma$,  on a compact space, $X$.  In this setting, the crossed product algebra, $C(X) \rtimes \Gamma$ can often be shown to be its own dual, as in Poincar\'e duality.  Since the duality presented in this paper is based on transversality coming from hyperbolic dynamics, it is natural to ask why this occurs.  The underlying idea is that the action of $\Gamma$ can be recoded so it has the same orbit structure as a single hyperbolic transformation.  Using this principle and the results in \cite{BowSer}, J. Spielberg showed that for certain Fuchsian groups acting on their boundary, the crossed product algebra was isomorphic to a Cuntz-Krieger algebra $O_A$.  This was extended by Laca and Spielberg, \cite{LS}, (see also C. Anantharaman-Delaroche \cite{Ana}) to show that the crossed product, $C(S^1)\rtimes SL(2,\Z)$ is isomorphic to a Ruelle algebra. One may ask how general a phenomenon this is.  According to Connes-Feldman-Weiss, \cite{CFW},  an amenable action is orbit equivalent, in the measure theoretic context,  to the action of a single transformation.  In the cases at hand, one gets much more than a orbit equivalence in the measure category, and the map it induces provides isomorphisms of the crossed product algebras with Ruelle algebras of associated hyperbolic dynamical systems.  It would be interesting to have a theory intermediate between measure theoretic orbit equivalence, which is naturally associated to von Neumann algebras, and topological flow equivalence, which is related to \csa s, where an equivalence would induce isomorphisms of Ruelle type algebras.

This line was pursued further by H. Emerson, who, without assuming that such crossed products are related to hyperbolic dynamical systems, was able to show   that if one takes $\Gamma$ to be a Gromov hyperbolic group, then there is an extension, analogous to the one above, whose $K$-homology class yields an isomorphism  $K^\ast(C(\partial \Gamma) \times \Gamma ) \cong K_{\ast+1}(C(\partial \Gamma) \times \Gamma )$ \cite{Eme2}. We will describe in Section \ref{sec:dual} how this clarifies a relation between the dynamical duality of the present paper and the Baum-Connes map for hyperbolic groups.

Building on ideas like these, the duality theory for general Smale spaces has played a role in the work of V. Nekrashevych \cite{Nek} on providing more precision to Sullivan's dictionary relating the dynamics of rational maps to that of Kleinian groups. Starting with a rational map, $f: \C \to \C$ , suitably restricted,  Nekrashevych constructs a self-similar group $\Gamma_f$, the iterated monodromy group of $f$ . This group has a limit set, ${\Lambda(\Gamma_f)}$, which admits a self map, $\lambda_f$, so that the pair $({\Lambda(\Gamma_f)}, \lambda_f)$ is topologically conjugate to $(f,\text{Julia}(f))$, the latter being $f$ restricted to its Julia set. He shows that the inverse limit $\underleftarrow{\lim} \{\Lambda(\Gamma_f), \lambda_f\}$ is a Smale space. Thus, one may study the stable and unstable Ruelle algebras and their duality in this context. Nekrashevych establishes that the unstable Ruelle algebra is Morita equivalent to an algebra associated to the iterated monodromy group, while the stable Ruelle algebra is Morita equivalent to the Deaconu-Renault algebra \cite{DM,Ren2}  associated to the rational map, $f$. One  may thus view the dynamical duality between the Ruelle algebras as relating the expanding dynamics of $f$ with the contracting dynamics associated to the action of the self-similar group on its limit set. To get closer to Sullivan's program, it would remain to relate the latter with the action of a Kleinian group on its limit set.  If the respective algebras are simple and purely infinite, as expected, then finding a Kleinian group whose algebras have the same K-theory as the self-similar group would provide support for Sullivan's dictionary.

The structure of the paper is as follows. Section 2 is an introduction to Smale spaces. We provide here more details on the technical aspects needed for later proofs. In Sections 3, we review the construction of the $C^\ast$-algebras associated to Smale spaces. Section 4 covers the $K$-theoretic duality we will use and describes various contexts where it arises naturally. In Section 5, the K-theory duality element is constructed. This is essentially a consequence of the transversality of the stable and unstable equivalence relations, and can be constructed in more generality. Indeed, the Mishchenko line bundle, used in the Baum-Connes assembly map, can be obtained this way. Section 6 is devoted to construction of the K-homology duality class, which provides the inverse on K-theory. Here, the hyperbolic nature of the dynamics plays a crucial role. Thus, it appears that the K-theory duality class exists in reasonable generality, but the existence of an inverse requires additional structure. This is precisely analogous to the difficulty in finding an inverse to the Baum-Connes assembly map using the Dirac-dual Dirac method. In Section 7, the proof of the main theorem will be completed and, finally, in Section 8 we will discuss open questions and possible extensions of the theory.

\section{Smale spaces}
\label{sec:Smale}

In this section, we provide a brief introduction to Smale spaces. The reader is also referred to \cite{Put1,Rue1}, but we will try to keep our treatment self-contained.

We assume that $(X, d)$ is a compact metric space and that $\varphi$ is a homeomorphism. The main gist of the definition is that, locally at a point $x$, $X$ is homeomorphic to the product of two subsets, denoted $X^{s}(x, \ep)$ and $X^{u}(x, \ep)$ and on these, the maps $\varphi$ and $\varphi^{-1}$, respectively, are contracting. For the uninitiated, it may be best to begin reading from Figure \ref{bracket_fig} to Definition \ref{defn:Smale_space} and then work backwards to the more rigorous aspects below.

 We assume the existence of constants, $ \ep_{X} > 0,  \lambda > 1$, and a map
\[ (x,y) \in X, d(x,y) \leq \ep_{X} \mapsto [x, y] \in X \]
satisfying a number of conditions. First, $[,]$ is jointly continuous on its domain of definition. Also, it satisfies
\begin{eqnarray*}
  \left[ x, x \right]     & =  & x, \\
  \left[ x, [ y, z] \right] &  = &  [ x, z],  \\
  \left[ [ x, y], z \right]  & = &  [ x,z ], \\
  \varphi[x, y] & =  & [ \varphi(x), \varphi(y)],
\end{eqnarray*}
for any $x, y, z$ in $X$, where both sides of the equality are defined. It follows easily from these axioms that $[x,y] = x$ if and only if $[y,x] =y$ and $[x,y] =y$ if and only if $[y,x] =x$. We define, for each $x$ in $X$ and $ 0 < \ep \leq \ep_{X}$, sets
\begin{eqnarray*}
 X^{s}(x, \ep) & = & \{ y \in X \mid d(x,y) \leq \ep, [y,x] =x \}, \\ 
X^{u}(x, \ep) & = & \{ y \in X \mid d(x,y) \leq \ep, [x,y] =x \}.
\end{eqnarray*}
It follows easily from the axioms that the map
\[ [,] : X^{u}(x, \ep) \times X^{s}(x, \ep) \rightarrow X \]
is a homeomorphism to its image, which is a neighbourhood of $x$ in $X$, provided $\ep \leq \ep_{X}/2$. The inverse map sends a point $z$ close to $x$ to the pair $( [z, x], [x, z])$. Moreover, as we vary $\ep$, the images form a neighbourhood base for the topology at $x$. To summarize, $X$ has a local product structure. We give a proof of the following simple result for completeness and because we will use it later in an essential way.

\begin{lemma} \label{lemma:bracket}
Given two points $x, y$ in $X$ and  $0 < \ep \leq \ep_{X}/2$, if the intersection $X^{s}(x, \ep) \cap X^{u}(y, \ep)$ is non-empty, then it is the single point $[x, y]$.
\end{lemma}

 \begin{proof}
Suppose that $z$ is in the intersection. This means that $[z, x]= x$ and $[y, z]=y$. It also means that $d(x, y) \leq d(x, z) + d(z, y) < 2 \ep \leq \ep_{X}$ so that $[x, y]$ is defined. It follows that $[x, y] = [[z, x], [y, z]] = [[z, x], z] =[z, z] =z$.
\end{proof}

It will probably help to have a picture of the bracket in mind, see figure \ref{bracket_fig}.

\begin{figure}[htb]
\begin{center}
\begin{tikzpicture}
\tikzstyle{axes}=[]
\begin{scope}[style=axes]
	\draw[<->] (-3,-1) node[left] {$X^s(x,\ep_X)$} -- (1,-1);
	\draw[<->] (-1,-3) -- (-1,1) node[above] {$X^u(x,\ep_X)$};
	\node at (-1.2,-1.4) {$x$};
	\node at (1.1,-1.4) {$[x,y]$};
	\pgfpathcircle{\pgfpoint{-1cm}{-1cm}} {2pt};
	\pgfpathcircle{\pgfpoint{0.5cm}{-1cm}} {2pt};
	\pgfusepath{fill}
\end{scope}
\begin{scope}[style=axes]
	\draw[<->] (-1.5,0.5) -- (2.5,0.5) node[right] {$X^s(y,\ep_X)$};
	\draw[<->] (0.5,-1.5) -- (0.5,2.5) node[above] {$X^u(y,\ep_X)$};
	\node at (0.7,0.2) {$y$};
	\node at (-1.6,0.2) {$[y,x]$};
	\pgfpathcircle{\pgfpoint{0.5cm}{0.5cm}} {2pt};
	\pgfpathcircle{\pgfpoint{-1cm}{0.5cm}} {2pt};
	\pgfusepath{fill}
\end{scope}
\end{tikzpicture}
\caption{The bracket map}
\label{bracket_fig}
\end{center}
\end{figure}
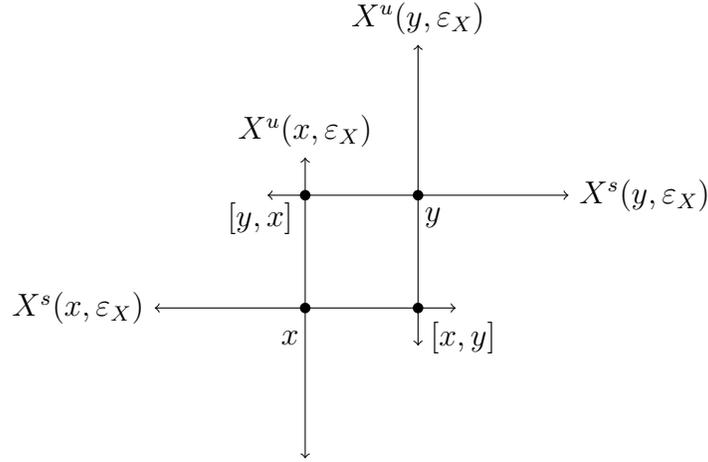

The final axiom is that, for $y,z$ in $X^{s}(x, \ep_{X})$, we have 
\[ d(\varphi(y),\varphi(z)) \leq \lambda^{-1} d(y,z), \]
and for $y,z$ in $X^{u}(x, \ep_{X})$, we have 
\[ d(\varphi^{-1}(y),\varphi^{-1}(z)) \leq \lambda^{-1} d(y,z). \]
That is, on the set $X^{s}(x, \ep_{X})$, $\varphi$ is contracting. It is tempting to say that  on $X^{u}(x, \ep_{X})$, $\varphi$ is expanding, but it is better to say that its inverse is contracting.

\begin{definition} \label{defn:Smale_space}
A Smale space is a compact metric space $(X, d)$ with a homeomorphism $\varphi$ such that there exist constants $ \ep_{X} > 0, \lambda > 1$ and map $[,]$ satisfying the conditions above.
\end{definition}

A Smale space $(X, d, \varphi)$ is said to be irreducible if the set of periodic points under $\varphi$ are dense and there is a dense $\varphi$-orbit.

The sets $X^{s}(x, \ep)$ and $X^{u}(x, \ep)$ are called the local contracting and expanding sets, respectively. We note for later convenience, that if $y$ is in $X^{s}(x, \ep)$, for some $x, y, \ep$, then for all $k \geq 0$,  $\varphi^{k}(y)$ is in $X^{s}(\varphi^{k}(x), \lambda^{-k} \ep)$. Similarly, if $y$ is in $X^{u}(x, \ep)$, for some $x, y, \ep$, then for all $k \geq 0$,  $\varphi^{-k}(y)$ is in $X^{u}(\varphi^{-k}(x), \lambda^{-k} \ep)$.

\begin{lemma} \label{lemma:epsilon'}
There is a constant $0 < \ep'_{X} \leq \ep_{X}/2$ such that, if $d(x, y) <  \ep'_{X} $, then 
both $d(x, [x,y]), d(y, [x,y]) <  \ep_{X}/2$ and hence $[x,y]$ is in $X^{s}(x, \ep_{X}/2)$ and in $X^{u}(y, \ep_{X}/2)$.
\end{lemma}

\begin{proof}
The functions $d(x, [x,y]), d(y, [x,y])$ are both defined on the set of pairs $(x, y)$ with $d(x, y) \leq \ep_{X}$, which is compact, and are continuous. Moreover, on the set where $x=y$, they have value zero. The existence of $\ep'_{X}$ satisfying the first conditions follows from uniform continuity. The last part follows from the definitions.
\end{proof}

We now define global stable and unstable equivalence relations on $X$. Given a point $x$ in $X$ we define the stable and unstable equivalence classes of $x$ by
\begin{eqnarray}
\nonumber
X^s(x) & = & \{y \in X | \lim_{n \rightarrow +\infty} d(\varphi^n(x),\varphi^n(y)) = 0\}, \\
\nonumber
X^u(x) & = & \{y \in X | \lim_{n \rightarrow +\infty} d(\varphi^{-n}(x),\varphi^{-n}(y)) = 0\}.
\end{eqnarray}
We will often denote stable equivalence by $x \sim_s x'$ and unstable equivalence by $y \sim_u y'$. To see the connection between these global stable and unstable sets, we note that, for any $x$ in $X$ and $\ep >0$,
$X^{s}(x, \ep) \subset X^{s}(x), X^{u}(x, \ep) \subset X^{u}(x)$. Moreover, $y$ in $X$ is in 
$X^{s}(x)$ (or $X^{u}(x)$) if and only if there exists $n \geq 0$ such that
$\varphi^{n}(y)$ is in $X^{s}(\varphi^{n}(x), \ep)$ (or $\varphi^{-n}(y)$ is in 
$X^{u}(\varphi^{-n}(x), \ep)$, respectively).

\section{$C^{*}$-algebras}
\label{sec:Algebras}

We describe the construction of $C^\ast$-algebras from a Smale space. In the introduction, we indicated that the $C^\ast$-algebras $S(X, \varphi)$ and $U(X, \varphi)$ are the $C^\ast$-algebras of the stable and unstable equivalence relations, respectively. This is correct in spirit, but for the purposes of this paper, it is a half-truth. We will find it much easier to work with equivalence relations which are equivalent to these (in the sense of Muhly, Renault and Williams \cite{MRW}) but which are \'{e}tale. To do this, we simply restrict to the stable and unstable equivalence classes of $\varphi$-invariant sets of points.  Some care must be taken because these unstable classes are endowed with a different (and more natural topology) than the relative topology of $X$. Specifically, we choose sets, $P$ and $Q$, consisting of periodic points and their orbits. We note that, at this point, there are no limitations on the sets $P$ and $Q$, however, later we require that they are distinct from one another. We shall then construct \'{e}tale groupoids of stable and unstable equivalence, $G^{s}(X, \varphi, Q)$ and $G^{u}(X, \varphi, P)$. The $C^\ast$-algebras of these groupoids will be denoted $S(X, \varphi, P)$ and  $U(X, \varphi, P)$, respectively.

Let $(X,d,\varphi)$ be a Smale space and let $P$ and $Q$ be finite sets of $\varphi$-invariant periodic points. Consider 
\[ X^{s}(P) = \bigcup_{p \in P} X^{s}(p), X^{u}(Q) = \bigcup_{q \in Q} X^{u}(q).  \]
The set $X^s(P)$ is endowed with locally compact and Hausdorff topology by declaring that the collection of sets $X^s(x,\ep)$, as $x$ varies over $X^s(P)$ and $0 < \ep < \ep_X$, forms a neighbourhood base. Similarly for $X^u(Q)$. The stable and unstable groupoids are then defined by
\begin{eqnarray*}
& \gs & = \{(v,w) | v \sim_s w \textrm{ and } v,w \in X^u(Q) \} \\
& \gu & = \{(v,w) | v \sim_u w \textrm{ and } v,w \in X^s(P) \}.
\end{eqnarray*}
Let $(v,w)$ be in $\gs$. Then for some $N \geq 0$, $\varphi^N(w)$ is in $X^s(\varphi^N(v),\ep_X^\prime)$. By the continuity of $\varphi$, we may find $\delta > 0$ such that
$\varphi^N(X^u(w,\delta)) \subset X^u(\varphi^N(w),\ep_X^\prime)$. A map $h^s: X^u(w,\delta) \rightarrow X^u(v, \ep_X/2)$ is defined by
\[ h^s(x) = \varphi^{-N}[\varphi^N(x),\varphi^N(v)]. \]
Moreover, it is the composition of three maps, $\varphi^N$, $[\, \cdot \, , \varphi^N(v)]$, and $\varphi^{-N}$, which are open on local unstable sets, and hence it is open. It is easy to verify that this map is a local homeomorphism and that interchanging the roles of $v$ and $w$ gives another local homeomorphism. Where composition of the two maps is defined it is the identity, in either order. Let 
\[ V^s(v,w,h^s,\delta) = \{ (h^s(x),x) | x \in X^u(w,\delta)\}. \]

It will help to have a picture of the map $h^s$:
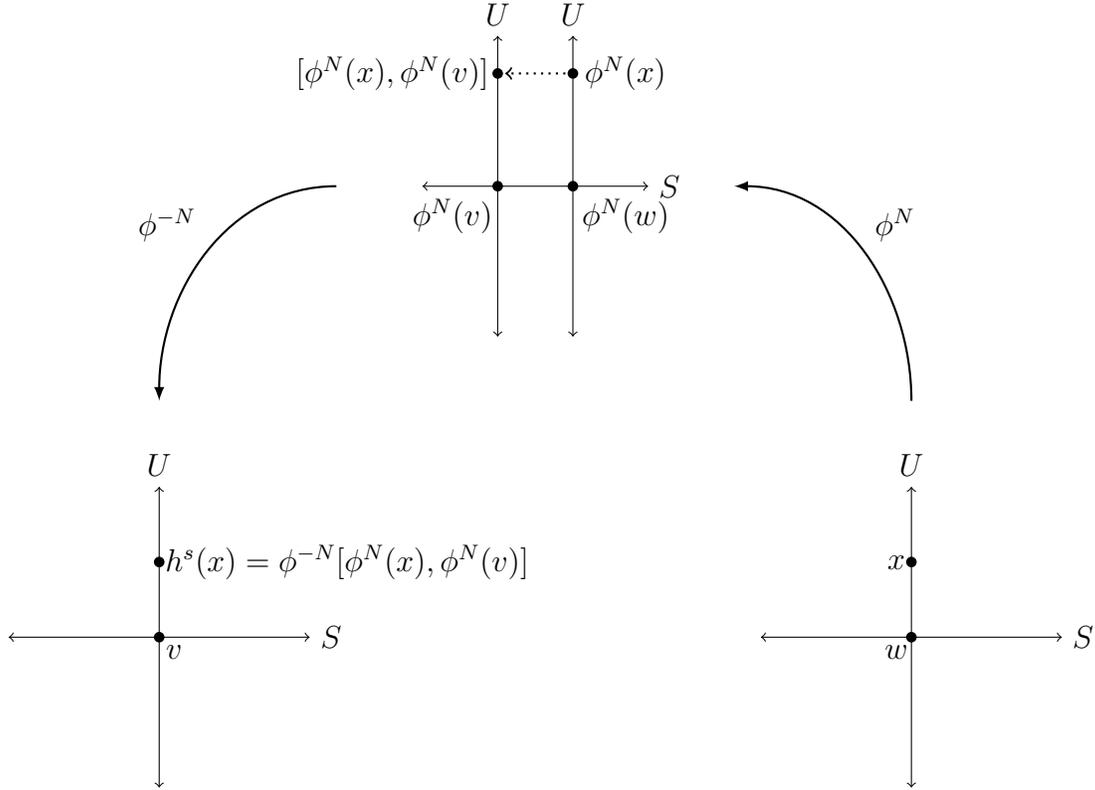
\begin{figure}[htb]
\begin{center}
\begin{tikzpicture}
\tikzstyle{axes}=[]
\begin{scope}[style=axes]
	\draw[<->] (3,0) -- (7,0) node[right] {$S$};
	\draw[<->] (5,-2) -- (5,2) node[above] {$U$};
	\node at (4.8,-0.2) {$w$};
	\node at (4.8,1) {$x$};
	\pgfpathcircle{\pgfpoint{5cm}{0cm}} {2pt};
	\pgfpathcircle{\pgfpoint{5cm}{1cm}} {2pt};
	\pgfusepath{fill}
\end{scope}
\begin{scope}[style=axes]
	\draw[<->] (-7,0) -- (-3,0) node[right] {$S$};
	\draw[<->] (-5,-2) -- (-5,2) node[above] {$U$};
	\node at (-4.8,-0.2) {$v$};
	\node at (-2.5,1) {$h^s(x)=\phi^{-N}[\phi^N(x),\phi^N(v)]$};
	\pgfpathcircle{\pgfpoint{-5cm}{0cm}} {2pt};
	\pgfpathcircle{\pgfpoint{-5cm}{1cm}} {2pt};
	\pgfusepath{fill}
\end{scope}
\begin{scope}[style=axes]
	\draw[<->] (-1.5,6) -- (1.5,6) node[right] {$S$};
	\draw[<->] (-.5,4) -- (-.5,8) node[above] {$U$};
	\draw[<->] (.5,4) -- (.5,8) node[above] {$U$};
	\node at (1.2,5.6) {$\phi^N(w)$};
	\node at (1.2,7.5) {$\phi^N(x)$};
	\node at (-1.1,5.6) {$\phi^N(v)$};
	\node at (-1.9,7.5) {$[\phi^N(x),\phi^N(v)]$};
	\pgfpathcircle{\pgfpoint{-.5cm}{6cm}} {2pt};
	\pgfpathcircle{\pgfpoint{.5cm}{6cm}} {2pt};
	\pgfpathcircle{\pgfpoint{-.5cm}{7.5cm}} {2pt};
	\pgfpathcircle{\pgfpoint{.5cm}{7.5cm}} {2pt};
	\pgfusepath{fill}
	\draw[->,thick,dotted] (0.4,7.5) -- (-0.4,7.5);
\end{scope}
\node (anchor1) at ( 2.5,6) {};
\node (anchor2) at ( 5,3) {}
	edge [->,>=latex,out=90,in=0,thick] node[auto,swap]{$\phi^N$} (anchor1);
\node (anchor3) at ( -2.5,6) {};
\node (anchor4) at ( -5,3) {}
	edge [<-,>=latex,out=90,in=180,thick] node[auto]{$\phi^{-N}$} (anchor3);
\end{tikzpicture}
\caption{The local homeomorphism $h^s: X^u(w,\delta) \rightarrow X^u(v,\delta)$}
\label{fig:local_homeo}
\end{center}
\end{figure}

\begin{lemma}\label{stable nbhd}
The collection of sets $V^s(v,w,h^s,\delta)$ as above forms a neighbourhood base for a topology on $\gs$ in which it is an \etale groupoid. 
\end{lemma}

We also remark that if $Q$ meets each irreducible component of $(X,d,\varphi)$, then this groupoid is unique up to the notion of equivalence given in \cite{MRW}. In \cite{PS}, it is shown that these groupoids are amenable.

Analogous results obviously hold for $\gu$.

Let  $C_c(\gs)$ denote the continuous functions of compact support on $\gs$, which is a complex linear space. A product and involution are defined on $C_c(\gs)$ as follows, for $f,g \in C_c(\gs)$ and $(x,y) \in \gs$,
\begin{eqnarray}
\nonumber
& f \cdot g(x,y) & = \sum_{z \sim_s x} f(x,z)g(z,y) \\
\nonumber
& f^\ast(x,y) & = \overline{f(y,x)}.
\end{eqnarray}
This makes $C_c(\gs)$ into a complex $\ast$-algebra. Any function in $C_c(\gs)$ may be written as a sum of functions, each having support in an element of the neighbourhood base described in \ref{stable nbhd}.

\begin{definition}\label{Algebras_SandU}
Let $(X,d,\varphi)$ be a Smale space and $P$ and $Q$ be $\varphi$-invariant sets of periodic points. We define $S(X,\varphi,Q)$ and $U(X,\varphi,P)$ to be the (reduced) $C^\ast$-algebras associated with the \'{e}tale groupoids $\gs$ and $\gu$, respectively. When no confusion will arise, we denote them by $S$ and $U$.
\end{definition}

We now want to define a canonical representation of these $C^\ast$-algebras. Let $X^h(P,Q)$ denote the set $X^s(P) \cap X^u(Q)$. Since $(X,d,\varphi)$ is assumed to be irreducible, $X^h(P,Q)$ is dense in $X$ \cite{Rue1}. Viewing $X^h(P,Q)$ as the set of $\gs$ equivalence classes of $P$, we consider the restriction of the regular representation of $S(X,\varphi,Q)$ to equivalence classes of $P$ in $\gs$. This means that we consider the Hilbert space $\h = \ell^2(X^h(P,Q))$, and for $a$ in $C_c(\gs)$, $\xi$ in $\h$, we define
\[ (a \xi)(x) = \sum_{(x,y) \in \gs} a(x,y)\xi(y) \qquad \qquad (x \text{ in } \h). \]
Notice that we suppress the notation for the representation of $C_c(\gs)$ as bounded operators on $\bh$.

Next, viewing the same set $X^h(P,Q)$ as the union of $\gu$-equivalence classes of points in $Q$, we can consider the the restriction of the regular representation of $U(X,\varphi,P)$ to the same Hilbert space $\h$. For $b$ in $C_c(\gu)$, $\xi$ in $\h$, we define
\[ (b \xi)(x) = \sum_{(x,y) \in \gu} b(x,y)\xi(y) \qquad \qquad (x \text{ in } \h). \]

For each $x$ in $X^h(P,Q)$, we let $\delta_x$ denote the function on $X^h(P,Q)$ taking value $1$ at $x$ and zero elsewhere. Of course, the set $\{ \delta_x \st x \in X^h(P,Q) \}$ forms a basis for the Hilbert space $\h = \ell^2(X^h(P,Q))$. The following two lemmas follow directly from the definitions. We omit the proofs.

\begin{lemma}\label{S_basic_function}
Let $V^s(v,w,N,\delta)$ be a basic open set in $\gs$ and suppose $a$ is a continuous compactly supported function on $V^s(v,w,h^s,\delta)$. For each $x$ in $X^h(P,Q)$, we have
\[ a\delta_x = a(h^s(x),x)\delta_{h^s(x)}, \]
where the right hand side is defined to be zero if $h^s(x)$ is not defined. Define $\sor(a) \subseteq X^u(w,\delta)$ to be the points for which $a$ is non-zero on its domain and define $\range(a) \subseteq X^u(v,\ep_X^\prime)$ to be the points in $X^u(v,\ep_X^\prime)$ for which $a(h^s(x),x)\delta_{h^u(x)}$ is non-zero. Observe that $a$ is zero on the orthogonal complement of $X^u(w,\delta)$.
\end{lemma}

\begin{lemma}\label{U_basic_function}
Let $V^u(v,w,N,\delta)$ be a basic open set in $\gu$ and suppose $b$ is a continuous compactly supported function on $V^u(v,w,h^u,\delta)$. For each $x$ in $X^h(P,Q)$, we have
\[ b\delta_x = a(h^u(x),x)\delta_{h^u(x)}, \]
where the right hand side is defined to be zero if $h^u(x)$ is not defined. Define $\sor(b) \subseteq X^s(w,\delta)$ to be the points for which $b$ is non-zero on its domain and define $\range(b) \subseteq X^s(v,\ep_X^\prime)$ to be the points in $X^s(v,\ep_X^\prime)$ for which $b(h^u(x),x)\delta_{h^s(x)}$ is non-zero. Observe that $b$ is zero on the orthogonal complement of $X^s(w,\delta)$.
\end{lemma}

We note that every element of either of the above $C^\ast$-algebras can be uniformly approximated by a finite sum of functions supported in a neighbourhood base sets. We also have a nice geometric picture of the Hilbert space $\h$ in the sense that a natural basis is parameterized by the points of $X^h(P,Q)$. In this spirit, the picture of the map $h^s$, given in Figure \ref{fig:local_homeo} on page \pageref{fig:local_homeo}, can now be viewed as a picture of the operator $a$, up to the value of the function at specific points.

The sets of periodic points $P$ and $Q$ are chosen to be $\varphi$-invariant and so $X^{s}(P)$ and $X^{u}(Q)$ are also. Moreover, it is clear  that $\varphi$ induces homeomorphisms of both these spaces. It is also clear that $\varphi \times \varphi$ induces automorphisms of $\gs$ and $\gu$. These in turn define automorphisms of the $C^{*}$-algebras, $S(X, \varphi, Q)$ and $U(X, \varphi, P)$, denoted $\alpha_{s}$ and $\alpha_{u}$, respectively. Specifically, suppose $a$ is a continuous compactly supported function on a basic set $V^s(v,w,h^s,\delta)$ and $x$ is in $X^h(P,Q)$, then we have
\[ \alpha_s(a)\delta_x = a(h^s \circ \varphi^{-1}(x),\varphi^{-1}(x)) \delta_{\varphi \circ h^s \circ \varphi^{-1}(x)}. \]
Similarly, if $b$ is a continuous compactly supported function on a basic set $V^u(v,w,h^u,\delta)$ and $x$ is in $X^h(P,Q)$, then we have
\[ \alpha_u(b)\delta_x = b(h^u \circ \varphi^{-1}(x),\varphi^{-1}(x)) \delta_{\varphi \circ h^u \circ \varphi^{-1}(x)}. \]

For the same reason $X^h(P,Q)$ is also $\varphi$-invariant and this implies that there is a canonical unitary operator $u$ on $\h$ defined by $u \xi = \xi \circ \varphi^{-1}$, for any $\xi$ in $\h$. We note that $u \delta_{x} = \delta_{\varphi(x)}$, for any $x$ in $X^h(P,Q)$. We also note, without proof, that 
\begin{eqnarray*} 
 u a u^{*}  &  =   \alpha_{s}(a),  & a \in S(X, \varphi, Q) \\
u b u^{*}  &  =   \alpha_{u}(b),  & b \in U(X, \varphi, P).
\end{eqnarray*}
These covariant pairs define crossed product $C^\ast$-algebras.

\begin{definition}\label{Stable Ruelle}
The stable and unstable Ruelle algebras, denoted by $R^S(X,\varphi,Q)$ and $R^U(X,\varphi,P)$, respectively, are the crossed products:
\[ R^S(X,\varphi,Q) = S(X,\varphi,Q) \rtimes_{\alpha_s} \mathbb{Z} \qquad \text{ and } \qquad R^U(X,\varphi,P) = U(X,\varphi,P) \rtimes_{\alpha_u} \mathbb{Z}. \]
\end{definition}

We remark that the Ruelle algebras, as defined here, are Morita equivalent to the Ruelle algebras defined in \cite{Put1}. Moreover, the Ruelle algebras were shown to be separable, simple, stable, nuclear, and purely infinite when $(X,d,\varphi)$ is irreducible \cite{PS}, hence they also satisfy the Universal Coefficient Theorem \cite{Tu2}. Moreover, according to the purely infinite case of Elliott's classification program, as developed by Kirchberg and Phillips, they are completely classified by their $K$-theory groups.

Due to the definitions of stable and unstable equivalence we also note that that $U(X, \varphi, Q) = S(X, \varphi^{-1}, Q)$ and $R^{U}(X, \varphi,Q) = R^{S}(X, \varphi^{-1}, Q)$.

\section{Noncommutative Duality}
\label{sec:dual}

In this section we will review basic facts about duality which we will need.  We will also describe some properties possessed by algebras which have duals and discuss various examples.  Much of this material can be found in \cite{KaPu1} or \cite{Eme2}.

\subsection{KK-theory}

Let $KK^0(A, B)$ denote the Kasparov $KK$-group for a pair of separable $C^{*}$-algebras $A$ and $B$.   Denote $C_0(0,1)$ by $\mathscr{S}$. Then $KK^1(A,B) = KK^0(A \otimes \mathscr{S},B)$. There are various ways of obtaining elements of $KK^0(A,B)$.  For example, any homomorphism, $h: A \to B$ determines an element $[h] \in KK^0(A,B)$.  If $A$ is nuclear then there is a natural isomorphism $KK^1(A,B) \cong Ext(A,B)$ \cite{Kas}, where $Ext(A,B)$ is the group of classes of \csa~  extensions of the form
\begin{equation}
 0 \to B \te \K \to \mathcal E \to A \to 0.
\end{equation}
Thus, an extension determines an element of $KK^1(A,B)$.  One can retrieve the ordinary K-theory and K-homology groups from KK-theory as
\begin{eqnarray}
 K_*(A) = KK^*(\C,A)\\
K^*(A) = KK^*(A,\C).
\end{eqnarray}

We will be using the Kasparov product,
\begin{equation}
\begin{diagram}
 KK^i(A,B \te D) \x KK^j(D \te A', B') & \rTo{\te_D} & KK^{i+j}(A \te A', B\te B').
\end{diagram}
\end{equation}
As usual,  indices are to be taken modulo 2.  In the course of proofs it will be necessary to be more precise about explicit expressions for products.

We refer to Connes' book, \cite{Con1}, p. 428, for a presentation of this material, and \cite{Bla} for a more complete treatment.

Let $1_D \in KK^0(D,D)$ denote the class determined by the identity homomorphism.  Then there are natural maps  $\tau_D: KK^i(A,B) \rightarrow KK^i(A\tensor D,B\tensor D)$  and $\tau^D: KK^i(A,B) \rightarrow KK^i(D\tensor A,D\tensor B)$ obtained via $x \mapsto x \tensor 1_D$ and $x \mapsto 1_D \tensor x$.

\subsection{Duality classes}

\begin{definition}\label{def:duality}
Let $A$ and $B$ be separable $C^\ast$-algebras. We say that $A$ and $B$ are {\em Spanier-Whitehead dual}, or just {\em dual}, if there are duality classes $\Delta \in KK^i(A \tensor B,\cp)$ and $\delta \in KK^i(\cp,A\tensor B)$ such that
\begin{eqnarray*}
\delta \te_B : K^j(B) \to K_{i+j}(A) \\
\te_A \Delta : K_j(A) \to K^{i+j}(B)
\end{eqnarray*}
yield inverse isomorphisms.
\end{definition}

The main criterion is the following theorem, first presented in Connes' book.  (c.f. \cite{Con1,KaPu1,Eme2})

\begin{theorem}
\label{criterion}
 Let $\Delta \in KK^i(A \tensor B,\cp)$ and $\delta \in KK^i(\cp,A\tensor B)$ be given, satisfying the two conditions,
\begin{eqnarray*}
 \delta \tensor_B \Delta & = & \, \, \, 1_A  \\
 \delta \tensor_A \Delta & = & (-1)^i 1_B.
\end{eqnarray*}
Then $\delta$ and $\Delta$ implement a duality between $A$ and $B$. 
\end{theorem}

Note that we are making use of the following standard conventions to make sense of the formulas in Theorem \ref{criterion}.  Let $\sigma : A \te B \to B \te A$ be the isomorphism interchanging the factors.
\begin{eqnarray*}
\delta \tensor_B \Delta & = & \sigma_{*}(\delta \tensor_B \sigma_{*}(\Delta))  \\
\delta \tensor_A \Delta & = & \sigma_{*}(\sigma^{*}(\delta) \tensor_A \Delta).
\end{eqnarray*}

\subsection{Bott periodicity and duality maps}
Because of our convention in the definition of $KK^1(A,B)$, for the sequel we will have to be explicit about how Bott periodicity fits into this for the sequel. We will also have to be more precise about the maps between K-groups induced by the duality elements.  

Let $\mathcal{T}$ denote the Toeplitz extension
\begin{diagram}
0 & \rTo & \mathcal{K}(\ell^2(\mathbb{N})) & \rTo & \mathcal{T} & \rTo & C(S^1) & \rTo & 0,
\end{diagram}
which determines an element of $KK^1(C(S^1),\cp)$. Observe that $\mathscr{S} \subset C(\mathbb{T}) $  and we denote the restriction of the Toeplitz extension to $\mathscr{S}$ by $\mathcal{T}_0$, which, by our conventions, is an element of $KK(\mathscr{S} \tensor \mathscr{S},\cp)$. Now if $\beta \in KK(\cp,\mathscr{S} \tensor \mathscr{S})$  is the Bott element, see 19.2.5 in \cite{Bla}, then we have
\[ \beta \tensor_{\mathscr{S} \tensor \mathscr{S}} \mathcal{T}_0 = 1_\cp \quad \textrm{ and } \quad \mathcal{T}_0 \tensor \beta = 1_{\mathscr{S} \tensor \mathscr{S}}, \]
(for this see Section $19.2$ in \cite{Bla}).   That is, $1_A \cong \tau^A(\mathcal{T}_0)$  and $1_B \cong \tau^B(\mathcal{T}_0)$.

In the present paper we will be working only with odd duality classes $\Delta \in KK^1(A \tensor B,\cp)$ and $\delta \in KK^1(\cp,A\tensor B)$.   We obtain maps between the various $K$-groups associated with $A$ and $B$ via the Kasparov product and we will need to be more precise about their relation to Bott periodicity.  To this end,
let $A$ and $B$ be $C^\ast$-algebras. Consider the homomorphisms $\Delta_i: K_i(A) \rightarrow K^{i+1}(B)$ and $\delta_i:K^{i}(B) \rightarrow K_{i+1}(A)$ defined by
\[ \begin{array}{cclc}
\Delta_0(x) & = & x \tensor_A \Delta & \qquad x \in K_0(A), \\
\Delta_1(x) & = & \beta \tensor_{\mathscr{S} \tensor \mathscr{S}}(\sigma_{*}(x \tensor_A \Delta)) & \qquad x \in K_1(A), \\ 
\delta_1(y) & = & \beta \tensor_{\mathscr{S} \tensor \mathscr{S}}(\delta \tensor_B y) & \qquad y \in K^1(B), \\
\delta_0(y) & = & \delta \tensor_B y & \qquad y \in K^0(B). \\
\end{array} \]

The compositions of these maps is described in the following result from \cite{Eme2}, which generalizes one from  \cite{KaPu1}.

\begin{theorem}[\cite{Eme2}]\label{dual_isom_B}
Let $A$ and $B$ be $C^\ast$-algebras. Suppose the classes $\Delta \in KK^1(A \tensor B,\cp)$ and $\delta \in KK^1(\cp,A\tensor B)$ satisfy the criterion in Theorem \ref{criterion}. Then,
\begin{eqnarray*}
\delta_{i+1} \circ \Delta_i & = & (-1)^i 1_{K_i(B)} \\
\Delta_{i+1} \circ \delta_i & = & (-1)^{i+1} 1_{K^{i}(A)}
\end{eqnarray*}
\end{theorem}

Interchanging the roles of $A$ and $B$ gives a similar result, and in either case we obtain isomorphisms $K_i(B) \cong K^{i+1}(A)$. 

\subsection{Consequences of duality}

In this section we will describe some algebraic consequences of an algebra having a dual.  Since this paper deals with Ruelle algebras associated to hyperbolic dynamical systems we will take advantage of the additional properties these algebras have.  In particular, they are separable, nuclear, purely infinite, and since they are algebras obtained from amenable groupoids, they satisfy the Universal Coefficient Theorem, \cite{RS} .  We will show that the Ruelle algebras are self-dual, hence satisfy a version of Poincar\'e duality.  This requires an additional hypothesis on the dynamical system which is very likely to hold in general.  Namely, we assume that $K_*(U(X,\varphi,P))$ and $K_*(S(X,\varphi,Q))$ are finite rank abelian groups.  Indeed, it will follow from results due to the second author on a special homology theory for Smale spaces \cite{Put2}, and is the analog of the rationality of the zeta function of such a system. 

 To start, we will assume that $A$ is separable and possesses the following properties. 
\begin{itemize}
 \item[a)] The algebra $A$ has an odd, separable, Spanier-Whitehead dual, $\D(A)$.
 \item[b)] The Universal Coefficient Theorem holds for $A$  and $\D(A)$, with K-homology in the middle.
\end{itemize}
It follows from this that,
\begin{itemize}
 \item[c)] The Universal Coefficient Theorem holds for $A$ and  $\D(A)$, with K-theory in the middle.
 \item[d)]  $K_* (A)$ and $K^*(A)$ are finitely generated groups. 
\end{itemize}

Statement (c) follows by applying duality to the Universal Coeffient Theorem for K-homology.  We will give a proof of (d).

\begin{proof}
First note that the Universal Coefficient Theorem and separability of $A$ imply that $\Hom(K_*(A),\Z)$ and $\Ext(K_*(A),Z)$ are both countable, as are the corresponding groups with $K^*(A)$ via duality.    Let $tK_*(A)$ denote the torsion subgroup of $K_*(A)$.  Applying $\Hom(\cdot ,\Z)$ to the sequence $$ 0 \to tK_*(A) \to  K_*(A) \to   K_*(A)/tK_*(A) \to 0$$ one deduces that $\Ext(K_*(A)/tK_*(A),\Z)$ is countable.  

It is shown in \cite{NR} that if a group $H$ is torsion free and $\Ext(H,\Z)$ is countable, then $H$ is free. Applying this to $K_*(A)/tK_*(A)$ one gets that $K_*(A) = tK_*(A) \oplus K_*(A)/tK_*(A)$.  It follows, since $\Hom(K_*(A),\Z)$ is countable, that $\Hom(K_*(A)/tK_*(A),\Z)$ is countable as well.  Thus, $K_*(A)/tK_*(A)$ must be finitely generated, or else it would be uncountable.  

Next one uses that for a torsion group $T$, $\Ext(T,\Z)$ is the Pontryagin dual of T, where T is given the discrete topology, \cite{Fuc}.  Thus, $\Ext(tK_*(A),\Z)$ is a compact topological group.  If it is infinite, then it is a perfect topological space, hence uncountable, so it must be finite.  Therefore, $K_*(A)$ is finitely generated.  Similarly, $K^*(A)$ is finitely generated.

\end{proof}

The Ruelle algebras $R^S(X,\varphi,Q)$ and $R^U(X,\varphi,P)$ are separable $C^*$-algebras associated to amenable groupoids. Hence, by \cite{Tu2}, they satisfy the Universal Coefficient Theorem and (d) applies.  Thus, we obtain
\begin{prop}
 The groups $K_*(R^S(X,\varphi,Q))$, $K_*(R^U(X,\varphi,P))$, $K^*(R^S(X,\varphi,Q))$, and $K^*(R^U(X,\varphi,P))$ are finitely generated. 
\end{prop}

\begin{prop}\label{isomorphism} 
If $\rank (K_0(R^S(X,\varphi,Q))) = \rank (K_1(R^S(X,\varphi,Q)))$, then $R^S(X,\varphi,Q) \cong R^U(X,\varphi,P)$.
\end{prop}
\begin{proof}
 The algebras $R^S(X,\varphi,Q)$ and $R^U(X,\varphi,P)$ satisfy the hypothesis of the Kirchberg-Phillips theorem, \cite{Phi}.  Thus we must only show that their K-theory groups are isomorphic.  Considering torsion first, note that, using duality and the Universal Coefficient Theorem,  

$$ tK_0(R^S(X,\varphi,Q)) \cong \Ext(K^1(R^S(X,\varphi,Q)), \Z) \cong tK^1(R^S(X,\varphi,Q)) \cong tK_0(R^U(X,\varphi,P)). $$  For the free part,  
\begin{eqnarray*}
 \rank (K_0(R^S(X,\varphi,Q))) & = & \rank (\Hom(K^0(R^S(X,\varphi,Q))),\Z) \\
 & = & \rank (\Hom(K_1(R^U(X,\varphi,P)),\Z)) \\
 & = &  \rank (K^1(R^U(X,\varphi,P))) = \rank (K_0(R^U(X,\varphi,P)))
\end{eqnarray*} 
A similar argument shows that $K_1(R^S(X,\varphi,Q)) \cong K_1(R^U(X,\varphi,P))$.  Thus, by the classification theorem, \cite{Phi}, we obtain that $R^S(X,\varphi,Q) \cong R^U(X,\varphi,P)$.
\end{proof}

\begin{remark}
 Note that $S(X,\varphi,Q)$ and $U(X,\varphi,P)$ are not isomorphic in general.
\end{remark}

\begin{cor}
 The algebras $R^S(X,\varphi,Q)$ and $R^U(X,\varphi,P)$ satisfy Poincar\'e duality.
 \end{cor}
\begin{proof}
 The isomorphism between $R^S(X,\varphi,Q)$ and $R^U(X,\varphi,P)$ is implemented by an element $\xi \in KK(R^S(X,\varphi,Q),R^U(X,\varphi,P))$.  Consider the duality classes \\ $\Delta \in KK^1(R^S(X,\varphi,Q)\otimes~R^U(X,\varphi,P),\C)$ and $\delta \in KK^1(\C,R^S(X,\varphi,Q)\otimes R^U(X,\varphi,P))$.  Let $\tilde\delta \in KK^1(\C,R^U(X,\varphi,P) \otimes R^U(X,\varphi,P))$ and $\tilde \Delta \in KK^1(R^U(X,\varphi,P) \otimes R^U(X,\varphi,P), \C)$ be defined by 
\begin{equation}
 \tilde \delta = \delta^\tau \otimes_{R^S(X,\varphi,Q)} \xi 
\end{equation}
and 
\begin{equation}
 \tilde \Delta = \xi^{-1} \otimes_{R^S(X,\varphi,Q)} \Delta^\tau  
\end{equation}
Then it follows from the properties of $\delta$ and $\Delta$ that, for $x \in K_0(R^U(X,\varphi,P))$ one has
\begin{equation}
 \tilde \delta \otimes_{R^U(X,\varphi,P)} (x \otimes_{R^U(X,\varphi,P)} \tilde \Delta) = x.
\end{equation}
\end{proof}

The question of when the hypothesis of Proposition \ref{isomorphism} holds will now be addressed.  

\begin{prop}
 Assume that $K_*(S(X,\varphi,Q))$ and $K_*(U(X,\varphi,P))$ have finite rank.  Then, one has 
\begin{eqnarray*}
\rank K_0(R^S(X,\varphi,Q)) & =  & \rank K_1(R^S(X,\varphi,Q)) \\
\rank K_0(R^U(X,\varphi,P)) & =  & \rank K_1(R^U(X,\varphi,P)).
\end{eqnarray*}
\end{prop}
\begin{proof}
 We will work out the case of $R^U(X,\varphi,P)$, the case of $R^S(X,\varphi,Q)$ being similar.  

Consider the Pimsner-Voiculescu sequence tensored with $\Q$,
\begin{diagram}
K_0(U(X,\varphi,P))\otimes\Q & \rTo^{1-\phi_0}  & K_0(U(X,\varphi,P))\otimes\Q & \rTo^\alpha K_0(R^U(X,\varphi,P))\otimes\Q \\
\uTo^{\delta_1}    &       &         & \dTo_{\delta_0} \\
K_1(R^U(X,\varphi,P))\otimes\Q & \lTo^\beta  & K_1(U(X,\varphi,P))\otimes\Q & \lTo^{1-\phi_1}  K_1(U(X,\varphi,P))\otimes\Q. \\
\end{diagram}
One checks directly that 
\begin{equation}
\label{K0}
 \begin{aligned}
  K_0(R^U(X,\varphi,P))\otimes\Q = \image \alpha \oplus  \cok \alpha \\ = \cok(1-\phi_0) \oplus \ker(1-\phi_1)
 \end{aligned}
\end{equation}
Similarly, 
\begin{equation}
\label{K1}
     K_1(R^U(X,\varphi,P))\otimes\Q = \cok(1-\phi_1) \oplus \ker(1-\phi_0).
 \end{equation}
But one obtains from,
\[ 0 \rightarrow \ker(1-\phi_0) \rightarrow K_0(U(X,\varphi,P))\otimes\Q \xrightarrow{1-\phi_0} K_0(U(X,\varphi,P))\otimes\Q \rightarrow \cok(1-\phi_0) \rightarrow 0 \]
that 
\begin{equation}
\begin{aligned}
 K_0(U(X,\varphi,P))\otimes\Q = \image(1-\phi_0) \oplus \cok(1-\phi_0) \\ = \image(1-\phi_0) \oplus \ker(1-\phi_0).
\end{aligned}
\end{equation}
Since $K_0(U(X,\varphi,P))$ is assumed to be of finite rank it follows that $\rank(\ker(1-\phi_0)) = \rank(\cok(1-\phi_0))$ and similarly $\rank(\ker(1-\phi_1)) = \rank(\cok(1-\phi_1))$.  Plugging these into (\ref{K0}) and (\ref{K1}) yields the result.
\end{proof}

\subsection{Dynamical duality and the Baum-Connes conjecture}

There are relations between the noncommutative duality we have been discussing and the Baum-Connes and Novikov conjectures in topology.  We consider a setting in which precise statements can be made.  Let $\Gamma$ be a torsion free, finitely presented group.  In this case, the Baum-Connes map, after tensoring with $\Q$, can be identified with 
\begin{equation}
 \mu \te \Q : K_*(B\Gamma) \te \Q \to K_*(C_r^*(\Gamma)) \te \Q,
\end{equation}
where the map $\mu \te \Q$ is obtained by taking Kasparov product with the class of the Mishchenko line bundle $\delta_{\Gamma} \in KK(\C, C_0(B\Gamma) \te C_r^*(\Gamma)) $.  Thus, rational version of the Baum-Connes conjecture here is equivalent to $C_0(B\Gamma)$ being (rationally) a Spanier-Whitehead dual to the noncommutative algebra $C_r^*(\Gamma)$.  Strictly speaking, it is the Baum-Connes conjecture as obtained by the Dirac-dual Dirac method, since just having the map $\mu \te \Q$ an isomorphism might not imply the existence of the K-homology duality class.

 In many cases where the Baum-Connes conjecture has been proved, use is made of nonpositive curvature.  This often provides what is needed to define a K-homology duality class which will give an inverse to the Baum-Connes map.  The duality we study in the present paper is based on hyperbolic dynamics.  On the other hand, in \cite{CGM}, and many other works, injectivity is proved for hyperbolic groups.  Indeed, there is a relation between the duality obtained here from hyperbolic dynamics and that from hyperbolic groups.  Results in this direction have been worked out in the thesis of Emerson and the paper by Higson, \cite{Hig1}.  We will state a result in a special case.

\begin{theorem}[\cite{Eme2,Hig1,Spi,Ana}]
 Let $\Gamma$ be a Fuchsian group with such that $D/\Gamma$ is a compact, oriented surface, and suppose that the boundary of $\Gamma$ is $S^1$.  Then one has the following commutative diagram.
\begin{diagram}
 KK^{\Gamma}(C_0(D),\C) & \rTo{\mu} & KK(\C, C^*_r(\Gamma))\\
\dTo{\partial}                    &           &   \dTo{i_*} \\
KK^1(C(\partial \Gamma) \rtimes \Gamma, \C) & \rTo{\mathcal{E} \te_{C(\partial \Gamma) \rtimes \Gamma}}  & KK(\C,C(\partial \Gamma) \rtimes \Gamma)
\end{diagram}
where $\mathcal{E} \in KK^1(C(\partial \Gamma) \rtimes \Gamma)\te C(\partial \Gamma) \rtimes \Gamma), \C) $ is the element constructed in \cite{Eme2}.
\end{theorem}

Although the lower map was motivated by the dynamical duality, the explicit connection is not apparent.  In the case at hand, and possibly in much more generality, it follows from \cite{Spi} or \cite{Ana}, that there is a Smale space $(X,\varphi)$ whose Ruelle algebras are isomorphic to the crossed product $C(\partial \Gamma) \rtimes \Gamma$.  This can be verified by using the Kirchberg-Phillips theorem and computing the K-theory groups.  What is not yet known is whether there are naturally defined isomorphisms for the vertical arrows making the diagram below commutative. 

\begin{diagram}
 KK^1(C(\partial \Gamma) \rtimes \Gamma, \C) & \rTo   & KK(\C,C(\partial \Gamma) \rtimes \Gamma)\\
\dTo{\cong} & & \dTo{\cong} \\
KK^1(R^U(X,\varphi,P), \C)) & \rTo & KK(\C,R^S(X,\varphi,Q)),
\end{diagram}
 Note that, once  an isomorphism is chosen for the left arrow, then there is a corresponding one defined for the right, but at present there is no geometrical way to obtain them. 

We note that S. Gorokhovsky, in his thesis, showed that Paschke duality also fits in the present setting.  Indeed, the Paschke dual of a \csa~ $A$ is a Spanier-Whitehead dual in the sense described here.

\section{The K-theory duality class}
\label{sec:delta}

We give a description of the duality class $\delta$ in $KK(\mathscr{S},R^S(X,\varphi,Q) \tensor R^U(X,\varphi,P))$ for the Ruelle algebras. Let $P$ and $Q$ be $\varphi$-invariant sets of periodic points with $P \cap Q = \emptyset$.

Before we begin with the technical details, let us explain the underlying idea of the construction. Consider the product groupoid $\gs \times \gu$ which is equivalent to the groupoid $\gh$ in the sense of Muhly, Renault, and Williams \cite{MRW}. Since the groupoid $\gh$ is an \etale groupoid with compact unit space, namely $X$ itself, its groupoid $C^\ast$-algebra, $H(X,\varphi)$, is unital. Thus, $K_0(H(X,\varphi))$ has a canonical element determined by the class of the identity. The above equivalence of groupoids implies that $S(X,\varphi,Q) \tensor U(X,\varphi,P)$ is Morita equivalent to $H(X,\varphi)$ and we construct a projection in $S(X,\varphi,Q) \tensor U(X,\varphi,P)$ corresponding to the class of the identity in $H(X,\varphi)$. For details regarding the Morita equivalence above see \cite{Put1}.

\begin{definition}
Suppose that $\mathcal{F} = \{ f_{1}, f_{2}, \ldots, f_{K} \}$ are continuous, non-negative functions on $X$  and $G = \{ g_{1}, \cdots, g_{K} \}$ 
is a subset of $X^h(P,Q) = X^s(P) \cap X^u(Q)$. For $0 < \epsilon \leq \ep'_{X}$, we say that $(\mathcal{F}, G)$ is an $\ep$-partition of $X$ if
\begin{enumerate}
\item the squares of the functions in $\mathcal{F}$ form a partition of unity in $C(X)$; that is,
\[ \sum_{k=1}^{K} f_{k}^{2}  =  1, \]
\item the elements of $G$ are all distinct,
\item the support of $f_{k}$ is contained in $B(g_{k}, \ep/2)$, for each $ 1 \leq k \leq K$.
\end{enumerate}
\end{definition}

\begin{lemma}\label{lemma:epsil_part}
There exists $(\mathcal{F}, G)$,  an $\ep'_{X}$-partition of $X$ such that
\[ (\mathcal{F} \circ \varphi^{-1}, \varphi(G))=(\{f_k \circ \varphi^{-1} \mid 1 \leq k \leq K\},\{\varphi(g_k) \mid 1\leq k \leq K\}) \]
is also an $\ep'_{X}$-partition of $X$. Moreover, $G$ can be chosen so that $G \cap \varphi(G) = \emptyset$.
\end{lemma}

\begin{proof}
Choose $\ep^\prime_X > \ep^\prime > 0$ small enough that, for any $x$ in $X$, $\varphi(B(x,\ep^\prime/2)) \subseteq B(\varphi(x),\ep^\prime_X/2)$. Let $U_x = B(x,\ep^\prime/4)$ so that $\{U_x\}_{x \in X}$ covers $X$. Since $X$ is compact there is a finite subcover, say $\{U_k\}_{k=1}^K$. Now a partition of unity subordinate to $\{U_k\}_{k=1}^K$ exists \cite{BT} and we define $\mathcal{F}=\{f_1,f_2, \cdots,f_K\}$ to be the square roots of these functions. Since $X^h(P,Q)$ is dense, we may choose points $g_k$ in $X^h(P,Q)$ to be within $\ep^\prime/4$ from the center of each ball $U_k$. Now the support of each function in $\mathcal{F}$ is still contained in a ball of radius $\ep^\prime/2$. Therefore, we have an $\ep'_{X}$-partition $(\mathcal{F}, G)$ such that $(\mathcal{F} \circ \varphi^{-1}, \varphi(G))$ is also an $\ep'_{X}$-partition.
\end{proof}

Now, for $0 < \ep \leq \ep^\prime_X$, let $(\mathcal{F},G)$ be an $\ep$-partition and define a function $p_G$ on $\gs \times \gu$ by setting
\[ p_G((x,x'), (y,y')) = f_i([x,y])f_j([x',y']), \]
for $(x,x') \in \gs, (y,y') \in \gu$, if, for some $i,j$, 
\[ x \in X^{u}(g_{i}, \ep), y \in X^{s}(g_{i}, \ep), x' \in X^{u}(g_{j}, \ep), y' \in X^{s}(g_{j}, \ep),  [x,y]=[x',y'] \] 
and to be zero otherwise. Notice  that  if a pair $i, j$ exist for a given $((x, x'), (y, y'))$, then it is unique, since 
$g_{i} = [y, x]$ and $g_{j} = [y', x']$.

\begin{lemma}\label{lemma:p_in_S_U}
Let $0<\ep \leq \ep'_{X}$ and let $(\mathcal{F}, G)$ be an $\epsilon$-partition. Then $p_G$ is in $S(X,\varphi,Q) \tensor U(X,\varphi,P)$.
\end{lemma}

\begin{proof}
Let us fix a pair $i,j$ and suppose there exists $(x,x^\prime) \in \gs$ and $(y,y^\prime) \in \gu$ such that
\[ x \in X^{u}(g_{i}, \ep), y \in X^{s}(g_{i}, \ep), x' \in X^{u}(g_{j}, \ep), y' \in X^{s}(g_{j}, \ep),  [x,y]=[x',y']. \]
We note that $[g_i,g_j]$ is defined and is stably equivalent to $g_i$ and unstably equivalent to $g_j$. By lemma \ref{stable nbhd} there are local homeomorphisms $h^s:X^u(g_i,\ep) \rightarrow X^u([g_i,g_j],\ep)$ and $h^u:X^s(g_j,\ep) \rightarrow X^s([g_i,g_j],\ep)$ defined by
\begin{eqnarray}
\nonumber
h^s(x) & = & [x,[g_i,g_j]]=[x,g_j] \\
\nonumber
h^u(y^\prime) & = & [[g_i,g_j],y^\prime] = [g_i,y^\prime]
\end{eqnarray}
It is immediate that if we let $x^\prime = h^s(x)$ and $y = h^u(y^\prime)$ then the points satisfy the conditions above. On the other hand, if $((x,x^\prime),(y,y^\prime))$ satisfy the conditions then we have
\begin{eqnarray}
\nonumber
x^\prime & = & [[x^\prime,y^\prime],x^\prime] = [[x,y],x^\prime] = [x,x^\prime] = [x,g_j] = h^s(x) \\
\nonumber
y & = & [y,[x,y]] = [y,[x^\prime,y^\prime]] = [y,y^\prime] = [g_i,y^\prime] = h^u(y^\prime).
\end{eqnarray}
This shows that points satisfying the conditions are realized by local homeomorphisms, one on the local unstable set of $g_i$ and one on the local stable set of $g_j$.

Set $\ep^\prime > 0$. Consider the function on $X^u(g_i,\ep) \times X^s(g_j,\ep)$ sending $(x,y^\prime)$ to $f_i([x,y])f_j([x',y'])$. It is clearly a continuous function of compact support so that it can be uniformly approximated within $\ep^\prime$ by a function of the form
\[ \sum_{k=1}^{K_{i,j}} a_{i,j,k}(x,x^\prime)b_{i,j,k}(y,y^\prime) \]
where, for each fixed $k$, we have $a_{i,j,k}$ in $C_c(\gs)$ and $b_{i,j,k}$ in $C_c(\gu)$. If there exists no $((x,x^\prime),(y,y^\prime))$ for a fixed $i,j$ we define the above sum to be zero. Now it follows that 
\[ \sum_{i,j} \sum_{k=1}^{K_{i,j}} a_{i,j,k} \tensor b_{i,j,k} \]
is within $\ep^\prime$ of $p_G$ in norm. This completes the proof.
\end{proof}

In the sequel, it will be convenient to have a description of the operator $p_G$ on the Hilbert space $\ell^2(X^h(P,Q)) \tensor \ell^2(X^h(P,Q))$, in terms of our usual basis, $\{ \delta_{w} \tensor \delta_{z} \mid w, z \in X^h(P,Q) \}$. We also introduce a standard convention that the bracket map returns the empty set when the bracket of two points is undefined. Of course, any operator applied to the dirac delta function of the empty set will return zero and we declare that any function of the empty set is also zero. This convention will simplify many of the upcoming formulations.

\begin{lemma}\label{lemma:p_op}
Let $0<\ep \leq \ep'_{X}$ and let $(\mathcal{F}, G)$ be an $\ep$-partition. Suppose $w,z$ are in $X^h(P,Q)$, then we have
\[ p_G (\delta_{w} \tensor \delta_{z}) = 
f_{k}([w,z]) \sum_{i=1}^{K} f_{i}([w,z])  \delta_{[w,g_{i}]} \tensor \delta_{[g_{i},z]} \]
if there exists a $1 \leq k \leq K$, such that $w \in X^u(g_{k}, \ep), z \in X^{s}(g_{k}, \ep)$ and is zero if there is no such $k$. (If the $k$ exists, it is unique, for given $w, z$. The expression on the right makes sense using our standard convention).
\end{lemma}

\begin{proof}
For any $x,y$ in $X^h(P,Q)$, we compute
\begin{eqnarray*}
(p_G(\delta_w \tensor \delta_z))(x,y) & = & \sum_{x^\prime \in X^{h}(x)} \sum_{y' \in X^{h}(y)} p_G((x,x'), (y,y')) \delta_{w}(x') \delta_{z}(y') \\
   &  =  & p_G((x,w), (y,z)) \\
   &  =  &  f_{i}([x,y])f_{k}([w,z]),
\end{eqnarray*}
provided
\[ x \in X^{u}(g_{i},\ep), y \in X^{s}(g_{i},\ep), w \in X^{u}(g_{k},\ep), z \in X^{s}(g_{k},\ep),  [x,y] =[w,z] \]
and zero otherwise. If there is no $k$ such that $w \in X^{u}(g_{k},\ep), z \in X^{s}(g_{k},\ep)$, then the conclusion holds. Let us continue under the assumption that there is such a $k$ (which must be unique, since $[z,w]=g_k$ and the bracket map is (locally) unique in a Smale space). If, for some $i$,  $[w,z]$ is not in the support of $f_i$, then for any $x, y$ 
as above for which $[x,y]=[w,z]$, we have $f_{i}([x,y]) = f_{i}([w,z]) = 0$. On 
the other hand, if $[w,z]$ is in the support of $f_i$, for some $i$, then 
\begin{eqnarray*}
x & = & [x, g_{i}] = [[x,y], g_{i}] = [[w,z],g_{i}] = [w, g_{i}] \\
y & = & [g_{i}, y] = [g_{i}, [x, y]] = [g_{i}, [w,z]] = [g_{i},z].
\end{eqnarray*}
That is, for a given $i$, the choice of $x, y$ is unique. For each such $i$, we have 
\[ (p_G (\delta_{w} \tensor \delta_{z}))([w,g_{i}], [g_{i}, z]) = f_{i}([w,z])f_{k}([w,z]), \]
and the left hand side is zero for all other values of $x,y$. The conclusion follows.
\end{proof}

\begin{lemma}\label{lemma:p_proj}
Let $0<\ep \leq \ep'_{X}$. If  $(\mathcal{F},G)$ is an $\ep$-partition, then $p_G$ is a projection. If $(\mathcal{F} \circ \varphi^{-1}, \varphi(G))$ is also an $\ep$-partition, then
\[ (u \tensor u) (p_G) (u^\ast \tensor u^\ast) = p_{\varphi(G)}. \]
\end{lemma}

\begin{proof}
To show that $p_G$ is a projection we use lemma \ref{lemma:p_op} to compute $(p_G)^{2} (\delta_{w} \tensor \delta_{z})$.
First of all, we have 
\[ p_G (\delta_{w} \tensor \delta_{z}) = f_{k}([w,z]) \sum_{i=1}^{K} f_{i}([w,z])  \delta_{[w,g_{i}]} \tensor \delta_{[g_{i},z]}, \]
if $w \in X^{u}(g_{k}, \ep), z \in X^{s}(g_{k}, \ep)$ and zero otherwise. We apply  $p_G$ again, taking it through the sum and looking at each term individually. That is, for fixed $ 1 \leq i \leq K$, we must consider, for what 
$l$ is $[w,g_{i}]$ in $X^{u}(g_{l}, \ep)$ and $[g_{i}, z]$ in $X^{s}(g_{l}, \ep)$. Since $[w,g_{i}]$ is clearly in $X^{u}(g_{i})$, this can only happen for $l=i$. Using this, we obtain
\begin{eqnarray*}
&& (p_G)^{2} (\delta_{w} \tensor \delta_{z}) = f_{k}([w, z]) \sum_{i=1}^{K} f_{i}([w, z]) p_G \delta_{[w,g_{i}]} \tensor \delta_{[g_{i},z]} \\
&& \qquad \quad =  f_{k}([w, z]) \sum_{i=1}^{K} f_{i}([w, z]) f_{i}([w, z]) \sum_{j=1}^{K} f_{j}([[w,g_{i}],[g_{i},z]]) \delta_{[[w, g_{i}],g_{j}]} \tensor \delta_{[g_{j},[g_{i},z]]} \\
&& \qquad \quad =  f_{k}([w, z]) \sum_{i=1}^{K} f_{i}([w, z])^{2} \sum_{j=1}^{K}  f_{j}([w, z]) \delta_{[w, g_{j}]} \tensor \delta_{[g_{j},z]} \\
&& \qquad \quad =  f_{k}([w, z])  \sum_{j=1}^{K}  f_{j}([w, z]) \delta_{[w, g_{j}]} \tensor \delta_{[g_{j},z]}  \\
&& \qquad \quad =  p_G (\delta_{w} \tensor \delta_{z}) 
\end{eqnarray*}
The second part of the proof is a computation and is omitted.
\end{proof}

From Lemma \ref{lemma:epsil_part}, we may find $\mathcal{F} = \{ f_{1}, \ldots, f_{K} \}, G = \{g_{1}, \ldots, g_{K} \}$ such that $(\mathcal{F},G)$ and $(\mathcal{F} \circ \varphi^{-1}, \varphi(G))$ are both $\ep'_{X}$-partitions of $X$ with $G \cap \varphi(G) = \emptyset$. Since $X^h(P,Q)$ contains no periodic points we know that neither does $G$. By lemmas \ref{lemma:p_in_S_U} and \ref{lemma:p_proj}, we have that both $p_G$ and $p_{\varphi(G)}$ are projections in $S(X,\varphi,Q) \tensor U(X,\varphi,P)$. For each $0 \leq s \leq 1$, consider
the collection 
\[ \mathcal{F}_{s} = \{ (1-s)^{1/2}f_{1}, \ldots, (1-s)^{1/2}f_{K}, s^{1/2} f_{1} \circ \varphi^{-1}, \ldots, s^{1/2} f_{K} \circ \varphi^{-1} \} \]
together with the set of points
\[ G_{s} = \{ g_{1}, \ldots, g_{K}, \varphi(g_{1}), \ldots, \varphi(g_{K}) \} \]
Clearly, $(\mathcal{F}_{s},G_s)$ is an $\ep'_{X}$-partition, for all $0 \leq s \leq 1$. The important features of $p_{G_s}$ are 
\begin{enumerate}
\item $p_{G_s}$ is  a path of projections in $S(X, \varphi, Q) \otimes U(X, \varphi, P)$,
\item $p_{G_s}$ arises from the $\ep'_X$-partition  $(\mathcal{F}_{s}, G_{s})$, for all $0 \leq s \leq 1$,
\item $p_{G_0} = p_G$ and 
\item $p_{G_1} = (u \tensor u) p_G (u^\ast \tensor u^\ast) = p_{\varphi(G)}$.
\end{enumerate}
Therefore $p_G$ and $p_{\varphi(G)}$ are homotopic projections in $S(X,\varphi,Q) \tensor U(X,\varphi,P)$.

Since $p_G$ and $p_{\varphi(G)}$ are homotopic, there is a partial isometry $v$ in $S(X,\varphi,Q) \tensor U(X,\varphi,P)$ with initial projection $v^\ast v = p_G$ and final projection $vv^\ast = p_{\varphi(G)}$. By lemma \ref{lemma:p_proj} we have that $(u \tensor u) p_G (u^\ast \tensor u^\ast) = p_{\varphi(G)}$ and it is easy to check that the operator $\varrho = (u \tensor u)p_G v^\ast$ has the property $\varrho^\ast \varrho = \varrho\varrho^\ast = p_{\varphi(G)}$. Note that the operator $\varrho$ is in $R^S(X,\varphi,Q) \tensor R^U(X,\varphi,P)$ but not in $S(X,\varphi,Q) \tensor U(X,\varphi,P)$ since $u \tensor u$ is in the former but not the latter.

We are now ready to define a $\ast$-homomorphism $\delta:\mathscr{S} \rightarrow R^S(X,\varphi,Q) \tensor R^U(X,\varphi,P)$. To do this, it suffices to define a partial isometry $V$ in $R^S(X,\varphi,Q) \tensor R^U(X,\varphi,P)$ with the same initial and final projection, $V^\ast V = VV^\ast$ is a projection. Then sending $z-1$ to $V - V^\ast V$ extends uniquely to such a map. 
(To see this, we simply note that $V + (1 - V^\ast V)$ is a unitary in the unitization of the range. So there is a unique $\ast$-homomorphism mapping $z$ in $C(S^{1})$ to $V$, whose restriction to $\mathscr{S} \cong C^\ast(z-1)$ is as claimed). Since $\varrho$ in $R^S(X,\varphi,Q) \tensor R^U(X,\varphi,P)$ has the property that $\varrho^\ast \varrho = \varrho\varrho^\ast = p_{\varphi(G)}$, we obtain the required $\ast$-homomorphism, which we denote by $\delta$.

\begin{definition}\label{class_delta}
The class $\delta$ in $KK(\mathscr{S},R^S(X,\varphi,Q) \tensor R^U(X,\varphi,P))$ is defined by the $\ast$ - homomorphism $\delta$ from $\mathscr{S}$ to $R^S(X,\varphi,Q) \tensor R^U(X,\varphi,P)$ which is uniquely determined by $\delta(z-1) = \varrho - \varrho^\ast \varrho$ where
\[ \varrho = (u \tensor u)p_G v^\ast. \]
\end{definition}

\section{The K-homology duality class}
\label{sec:Delta}

For the K-homology duality class we construct an extension of $R^S(X,\varphi,Q) \tensor R^U(X,\varphi,P)$.

Recall that $\h = \ell^2(X^h(P,Q))$. From section \ref{sec:Algebras}, we have representations of $S(X,\varphi,Q)$, $U(X,\varphi,P)$, $R^S(X,\varphi,Q)$ and $R^U(X,\varphi,P)$ as bounded operators on $\h$.

The first observation is that, since these algebras are represented on the same Hilbert space, we can consider how operators coming from $S(X,\varphi,Q)$ and $U(X,\varphi,P)$ interact on $\h$. The following three Lemmas elucidate these interactions. We have used a hyperbolic toral automorphism to illustrated the main concepts in Lemma \ref{lemma:SU_cmpct} and Lemma \ref{lemma:asymp_abel} in Figures \ref{fig:HTA_ab_compact} and \ref{fig:HTA_alpha_support}, respectively.

For these three Lemmas let us fix the following elements. Assume that $a$ in $S(X,\varphi,Q)$ and $b$ in $U(X,\varphi,P)$ are both supported on basic sets; that is, for $v,w \in X^u(Q)$ and $v',w' \in X^s(P)$, let the support of $a$ be $V^s(v,w,h^s,\delta)$ and the support of $b$ be $V^u(v',w',h^u,\delta')$. Note that $\sor(a) \subseteq X^u(w,\delta)$ and $\range(a) \subseteq X^u(v,\delta)$, and $\sor(b) \subseteq X^s(w^\prime,\delta^\prime)$ and $\range(b) \subseteq X^s(v^\prime,\delta^\prime)$. See lemma \ref{S_basic_function} for further details.

\begin{lemma}[\cite{Put1}] \label{lemma:SU_cmpct}
If $a$ is in $S(X,\varphi,Q)$ and $b$ is in $U(X,\varphi,P)$, then $ab$ and $ba$ are compact operators on $\h$.
\end{lemma}

\begin{proof}
We compute, for $x$ in $X^h(P,Q)$,
\begin{eqnarray*}
a \cdot b \, \delta_x & = & a(h^s \circ h^u(x),h^u(x)) b(h^u(x),x) \delta_{h^s \circ h^u(x)}
\end{eqnarray*}
if $x \in X^s(w',\delta')$, $h^u(x) \in X^s(v',\delta')$,  $h^u(x) \in X^u(w,\delta)$, and $h^s \circ h^u(x) \in X^u(v,\delta)$. Otherwise the product is zero. In particular, the product is zero unless $\range(b) \cap \sor(a)$ is non-zero. However, uniqueness of the bracket implies that a local stable set and a local unstable set have non-trivial intersection at one point, at most. Whence, the product is zero unless $X^s(v',\delta')$ and $X^u(w,\delta)$ intersect and if they do the product is a rank one operator. Now finite sums of operators with supports as above form a dense set and therefore we obtain the compact operators by taking limits. Taking adjoints gives that $b \cdot a$ is also compact.
\end{proof}

\begin{figure}[htb]
\begin{center}
\scalebox{0.69}{
\begin{tikzpicture}
\tikzstyle{axes}=[]
\begin{scope}[style=axes]
	\draw[-] (0,0) node[below] {$(0,0)$} -- (0,10) node[above] {$(0,1)$};
	\draw[-] (0,10) -- (10,10);
	\draw[-] (10,0) node[below] {$(1,0)$} -- (10,10) node[above] {$(1,1)$};
	\draw[-] (0,0) -- (10,0);
\end{scope}
\begin{scope}[style=axes]
	\draw[|-|] (1.7,2.3) -- (3.7,3.53);
	\node at (4.6,3.6) {$X^u(w,\delta)$};
	\draw[|-|] (4.42,7.885) -- (6.42,9.115);
	\node at (7.3,9.2) {$X^u(v,\delta)$};
	\draw[-,dashed,shorten <=5pt] (2.7,2.915) -- (4.50,0);
	\draw[->,dashed,shorten >=5pt] (4.50,10) -- (5.42,8.5);
	\node at (4.4,0.8) {$h^s$};
	\pgfpathcircle{\pgfpoint{2.7cm}{2.915cm}} {2pt};
	\pgfusepath{fill}
	\node at (2.5,3.2) {$w$};
	\pgfpathcircle{\pgfpoint{5.42cm}{8.5cm}} {2pt};
	\pgfusepath{fill}
	\node at (5.55,8.15) {$v$};
	\pgfpathcircle{\pgfpoint{1.95cm}{2.45cm}} {2pt};
	\pgfusepath{fill}
	\node at (1.6,1.8) {$h^u(x)$};
	\pgfpathcircle{\pgfpoint{4.67cm}{8.03cm}} {2pt};
	\pgfusepath{fill}
	\node at (3.9,8.4) {$h^s \circ h^u(x)$};
\end{scope}[style=axes]
\begin{scope}[style=axes]
	\draw[|-|] (2.24,2) -- (1,4);
	\node at (0.9,4.3) {$X^s(v^\prime,\delta^\prime)$};
	\draw[|-|] (8.24,5.71) -- (7,7.71);
	\node at (8.5,5.2) {$X^s(w^\prime,\delta^\prime)$};
	\draw[<-,dashed,shorten >=5pt,shorten <=5pt] (1.62,3) -- (7.62,6.71);
	\node at (3.9,4.9) {$h^u$};
	\pgfpathcircle{\pgfpoint{1.62cm}{3cm}} {2pt};
	\pgfusepath{fill}
	\node at (1.2,2.8) {$v^\prime$};
	\pgfpathcircle{\pgfpoint{7.62cm}{6.71cm}} {2pt};
	\pgfusepath{fill}
	\node at (7.9,7.0) {$w^\prime$};
	\pgfpathcircle{\pgfpoint{7.95cm}{6.16cm}} {2pt};
	\pgfusepath{fill}
	\node at (8.2,6.4) {$x$};
\end{scope}[style=axes]
\end{tikzpicture}}
\caption{Hyperbolic toral automorphism: ab is a compact operator}
\label{fig:HTA_ab_compact}
\end{center}
\end{figure}
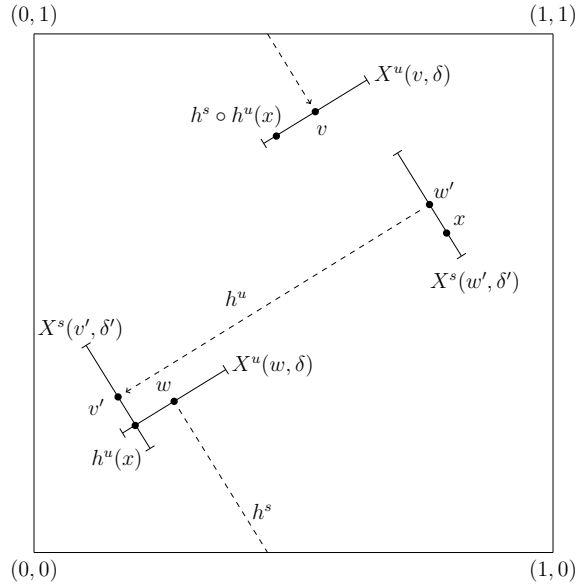

\begin{lemma}\label{SU_eventually_zero}
If $a$ is in $S(X,\varphi,Q)$ and $b$ is in $U(X,\varphi,P)$, then 
\[ \lim_{n \to +\infty} \alpha_s^{-n}(a)\cdot b = 0 \qquad \textrm{ and } \qquad \lim_{n \rightarrow +\infty} b \cdot \alpha_s^{-n}(a) = 0. \]
\end{lemma}

\begin{proof}
We first aim to show that there exists $N$ in $\mathbb{N}$ such that, for all $n \geq N$, we have $\alpha^{-n}(a) \cdot b = 0$. Indeed, since
\[ \alpha^{-n}_s(a) \, \delta_z = a(h^s \circ \varphi^{n}(z),\varphi^{n}(z)) \delta_{\varphi^{-n} \circ h^s \circ \varphi^{n}(z)}, \]
we see that the support of $\alpha_s^{-n}(a)$ is $V^s(\varphi^{-n}(v),\varphi^{-n}(w),\varphi^{-n} \circ h^s \circ \varphi^{-n}, \lambda^{-n}\delta)$. Therefore, it follows that $\sor(\alpha_s^{-n}(a)) \subseteq X^u(\varphi^{-n}(w),\lambda^{-n}\delta)$ and $\range(\alpha_s^{-n}(a)) \subseteq X^u(\varphi^{-n}(v),\lambda^{-n}\delta)$. That is, the support of $a$ is being exponentially contracted by repeated application of $\alpha_s$. Moreover, we compute
\begin{eqnarray*}
\alpha_s^{-n}(a) \cdot b \, \delta_z & = & a(h^s \circ \varphi^{n} \circ h^u(z),\varphi^{n} \circ h^u(z)) b(h^u(z),z) \delta_{\varphi^{-n} \circ h^s \circ \varphi^{n} \circ h^u(z)}
\end{eqnarray*}
if $z \in X^s(w',\delta')$, $h^u(z) \in X^s(v',\delta')$,  and $h^u(z) \in X^u(\varphi^{-n}(w),\lambda^{-n}\delta)$. It is zero otherwise.  
Now set $\ep > 0$ small enough that $X^u(Q,\ep) \cap X^s(v',\delta')= \varnothing$, we know this is possible since $v'$ is in $X^s(P)$ while no point in $Q$ is in $X^s(P)$ since $P$ and $Q$ are mutually distinct and $\varphi$-invariant. Given $\ep > 0$, we can find an $N$ in $\mathbb{N}$, such that $X^u(\varphi^{-n}(w),\lambda^{-n}\delta) \subset X^u(Q,\ep)$ for all $n \geq \mathbb{N}$. This implies that, for all $n \geq \mathbb{N}$, we have $\alpha^{-n}(a) \cdot b = 0$. Now the general result follows since elements of $S(X,\varphi,Q)$ and $U(X,\varphi,P)$ are norm limits of linear combinations of elements with the above form. A similar argument gives the result for $b \cdot \alpha^{-n}(a)$.
\end{proof}

\begin{lemma}[\cite{Put1} \cite{Kil}] \label{lemma:asymp_abel}
For any  $a$  in $S(X,\varphi,Q)$ and  $b$  in $U(X,\varphi,P)$, we have
\begin{eqnarray*}
&& \lim_{n \rightarrow \infty} \| \alpha_s^n(a) b - b  \alpha_s^n(a) \| = 0, \\
&& \lim_{n \rightarrow \infty} \| \alpha_s^n(a) \alpha_u^{-n}(b) - \alpha_u^{-n}(b)\alpha_s^n(a) \| = 0. \\
\end{eqnarray*}
\end{lemma}

\begin{proof}
We shall prove the second equality only, from which the first is easily deduced. Set $\ep > 0$. We compute
\[ \alpha_s^n(a)\cdot \alpha_u^{-n}(b) \, \delta_z = \]
\[ a(h^s \circ \varphi^{-2n} \circ h^u \circ \varphi^{n}(z),\varphi^{-2n} \circ h^u \circ \varphi^{n}(z)) b(h^u \circ \varphi^{n}(z),\varphi^{n}(z)) \delta_{\varphi^{n}\circ h^s \circ \varphi^{-2n} \circ h^u \circ \varphi^{n}(z)} \]
\[ \alpha_u^{-n}(b) \cdot \alpha_s^{n}(a) \, \delta_z = \]
\[ b(h^u \circ \varphi^{2n} \circ h^s \circ \varphi^{-n}(z),\varphi^{2n} \circ h^s \circ \varphi^{-n}(z)) a(h^s \circ \varphi^{-n}(z),\varphi^{-n}(z)) \delta_{\varphi^{-n}\circ h^u \circ \varphi^{2n} \circ h^s \circ \varphi^{-n}(z)}.\]
Moreover, lemma $2.2$ in \cite{Put1} states that, there exists $N$ such that for all $n \geq N$ we have,
\[ \varphi^{n}\circ h^s \circ \varphi^{-2n} \circ h^u \circ \varphi^{n}(z) = \varphi^{-n}\circ h^u \circ \varphi^{2n} \circ h^s \circ \varphi^{-n}(z).\]

Now, suppose we are given $z$ in $X^h(P,Q)$ such that $\varphi^{-n}(z) \in \sor(a)$ and $\varphi^{n}(z) \in \sor(b)$. We may define the following points:
\begin{eqnarray*}
x_1 & = & z \\
x_2 & = & \varphi^n \circ h^s \circ \varphi^{-n}(z) \\
x_3 & = & \varphi^{n}\circ h^s \circ \varphi^{-2n} \circ h^u \circ \varphi^{n}(z) = \varphi^{-n}\circ h^u \circ \varphi^{2n} \circ h^s \circ \varphi^{-n}(z) \\
x_4 & = & \varphi^{-n} \circ h^u \circ \varphi^n(z).
\end{eqnarray*}
In fact, given $\ep_1 > 0$ and any $z$ satisfying the above conditions, we can set $N$ sufficiently large that we have, for all $n \geq N$:
\begin{eqnarray*}
x_2 \in X^s(x_1,\ep_1) & & x_4 \in X^u(x_1,\ep_1), \\
x_1 \in X^s(x_2,\ep_1) & & x_3 \in X^u(x_2,\ep_1), \\
x_4 \in X^s(x_3,\ep_1) & & x_2 \in X^u(x_3,\ep_1), \\
x_3 \in X^s(x_4,\ep_1) & & x_1 \in X^u(x_4,\ep_1).
\end{eqnarray*}
We have illustrated the relationship between these points for the hyperbolic toral automorphism in Figure \ref{fig:HTA_alpha_support} on page \pageref{fig:HTA_alpha_support}. Since $a$ and $b$ are taking basis vectors to basis vectors, we have 
\[ \| \alpha_s^n(a)\cdot \alpha_u^{-n}(b) - \alpha_u^{-n}(b) \cdot \alpha_s^{n}(a) \| \]
\[= \sup_{z} | a(\varphi^{-n}(x_3),\varphi^{-n}(x_4)) b(\varphi^{n}(x_4),\varphi^{n}(x_1)) - b(\varphi^{n}(x_3),\varphi^{n}(x_2)) a(\varphi^{-n}(x_2),\varphi^{-n}(x_1)) |. \]
Now $a$ and $b$ are uniformly continuous so we may choose $N$ large enough that the above condition is satisfied (the two versions of $x_3$ are equal) and so that $\ep_1$ is sufficiently small that we have, for all $n \geq N$,
\begin{eqnarray*}
&& |a(\varphi^{-n}(x_3),\varphi^{-n}(x_4)) - a(\varphi^{-n}(x_2),\varphi^{-n}(x_1))| < \frac{\ep}{2\|b\|} \\
&& |b(\varphi^{n}(x_4),\varphi^{n}(x_1)) - b(\varphi^{n}(x_3),\varphi^{n}(x_2))| < \frac{\ep}{2\|a\|}.
\end{eqnarray*}
Now we compute
\begin{eqnarray*}
&& | a(\varphi^{-n}(x_3),\varphi^{-n}(x_4)) b(\varphi^{n}(x_4),\varphi^{n}(x_1)) - a(\varphi^{-n}(x_2),\varphi^{-n}(x_1))b(\varphi^{n}(x_3),\varphi^{n}(x_2)) | \\
&& = | a(\varphi^{-n}(x_3),\varphi^{-n}(x_4)) b(\varphi^{n}(x_4),\varphi^{n}(x_1)) - a(\varphi^{-n}(x_3),\varphi^{-n}(x_4))b(\varphi^{n}(x_3),\varphi^{n}(x_2)) \\
&& \quad + a(\varphi^{-n}(x_3),\varphi^{-n}(x_4))b(\varphi^{n}(x_3),\varphi^{n}(x_2)) - a(\varphi^{-n}(x_2),\varphi^{-n}(x_1))b(\varphi^{n}(x_3),\varphi^{n}(x_2)) | \\
&& \leq |a(\varphi^{-n}(x_3),\varphi^{-n}(x_4))||b(\varphi^{n}(x_4),\varphi^{n}(x_1)) - b(\varphi^{n}(x_3),\varphi^{n}(x_2)) | \\
&& \quad + |a(\varphi^{-n}(x_3),\varphi^{-n}(x_4)) - a(\varphi^{-n}(x_2),\varphi^{-n}(x_1))||b(\varphi^{n}(x_3),\varphi^{n}(x_2))| \\
&& < \frac{\|a\| \ep }{2 \|a\|} + \frac{\|b\| \ep }{2 \|b\|} = \ep.
\end{eqnarray*}
Therefore, we have shown that
\[ \lim_{n \rightarrow \infty} \| \alpha_s^n(a) \alpha_u^{-n}(b) - \alpha_u^{-n}(b)\alpha_s^n(a) \| = 0 \]
which completes the proof.
\end{proof}

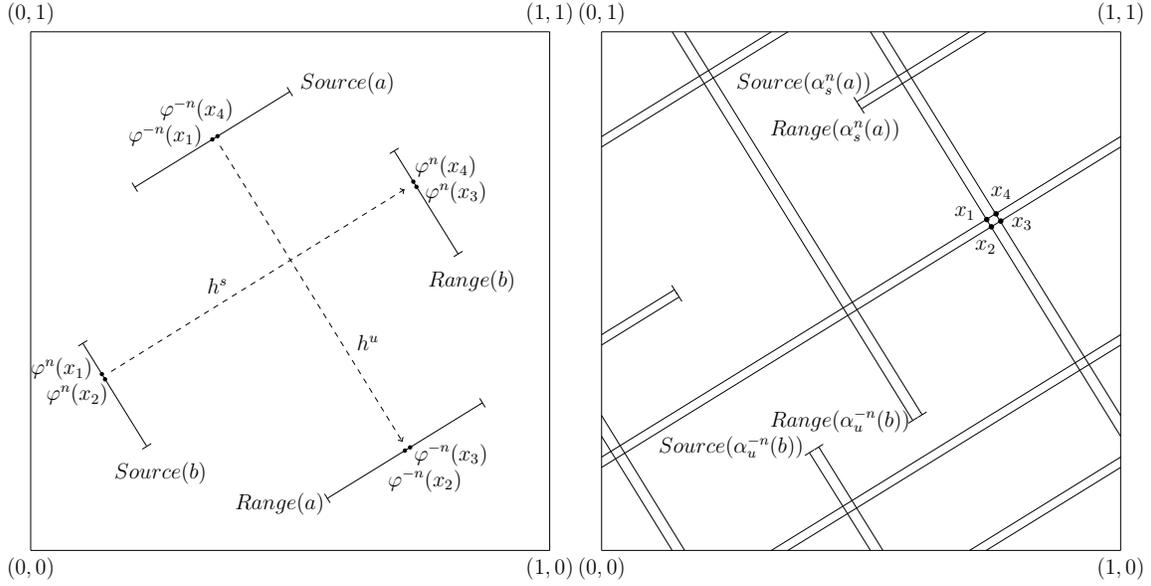
\begin{figure}[htb]
\begin{center}
\scalebox{0.69}{
\begin{tikzpicture}
\tikzstyle{axes}=[]
\begin{scope}[style=axes]
	\draw[-] (0,0) node[below] {$(0,0)$} -- (0,10) node[above] {$(0,1)$};
	\draw[-] (0,10) -- (10,10);
	\draw[-] (10,0) node[below] {$(1,0)$} -- (10,10) node[above] {$(1,1)$};
	\draw[-] (0,0) -- (10,0);
\end{scope}
\begin{scope}[style=axes]
	\draw[|-|] (2,7) -- (5,8.85);
	\node at (6.1,9.0) {$\sor(a)$};
	\draw[|-|] (5.71,1) -- (8.71,2.85);
	\node at (4.8,0.9) {$\range(a)$};
	\draw[->,dashed] (3.65,7.8) -- (7.16,2.1);
	\node at (6.5,4) {$h^u$};
\end{scope}[style=axes]
\begin{scope}[style=axes]
	\draw[|-|] (2.24,2) -- (1,4);
	\node at (2.5,1.5) {$\sor(b)$};
	\draw[|-|] (8.24,5.71) -- (7,7.71);
	\node at (8.5,5.2) {$\range(b)$};
	\draw[->,dashed] (1.55,3.45) -- (7.2,6.96);
	\node at (3.6,5.1) {$h^s$};
\end{scope}[style=axes]
\begin{scope}[style=axes]
	\pgfpathcircle{\pgfpoint{3.5cm}{7.925cm}} {1.2pt};
	\pgfusepath{fill}
	\node at (2.6,8) {$\varphi^{-n}(x_1)$};
	\pgfpathcircle{\pgfpoint{3.6cm}{7.987cm}} {1.2pt};
	\pgfusepath{fill}
	\node at (3.2,8.5) {$\varphi^{-n}(x_4)$};
	\pgfpathcircle{\pgfpoint{7.21cm}{1.925cm}} {1.2pt};
	\pgfusepath{fill}
	\node at (7.6,1.35) {$\varphi^{-n}(x_2)$};
	\pgfpathcircle{\pgfpoint{7.31cm}{1.987cm}} {1.2pt};
	\pgfusepath{fill}
	\node at (8.1,1.85) {$\varphi^{-n}(x_3)$};
\end{scope}
\begin{scope}[style=axes]
	\pgfpathcircle{\pgfpoint{1.372cm}{3.4cm}} {1.2pt};
	\pgfusepath{fill}
	\node at (0.6,3.5) {$\varphi^{n}(x_1)$};
	\pgfpathcircle{\pgfpoint{1.433cm}{3.3cm}} {1.2pt};
	\pgfusepath{fill}
	\node at (0.9,3.0) {$\varphi^{n}(x_2)$};
	\pgfpathcircle{\pgfpoint{7.372cm}{7.11cm}} {1.2pt};
	\pgfusepath{fill}
	\node at (8.0,7.4) {$\varphi^{n}(x_4)$};
	\pgfpathcircle{\pgfpoint{7.433cm}{7.01cm}} {1.2pt};
	\pgfusepath{fill}
	\node at (8.2,6.9) {$\varphi^{n}(x_3)$};
\end{scope}

\begin{scope}[style=axes]
	\draw[-] (11,0) node[below] {$(0,0)$} -- (11,10) node[above] {$(0,1)$};
	\draw[-] (11,10) -- (21,10);
	\draw[-] (21,0) node[below] {$(1,0)$} -- (21,10) node[above] {$(1,1)$};
	\draw[-] (11,0) -- (21,0);
\end{scope}
\begin{scope}[style=axes]
	\draw[-|] (18.41,10) -- (16,8.51);
	\draw[-|] (18.09,10) -- (15.91,8.65);
	\draw[-] (18.41,0) -- (21,1.6);
	\draw[-] (18.09,0) -- (21,1.8);
	\draw[-] (11,1.6) -- (21,7.78);
	\draw[-] (11,1.8) -- (21,7.98);
	\draw[-] (11,7.78) -- (14.59,10);
	\draw[-] (11,7.98) -- (14.27,10);
	\draw[-] (14.59,0) -- (21,3.96);
	\draw[-] (14.27,0) -- (21,4.16);
	\draw[-|] (11,4.16) -- (12.41,5.03);
	\draw[-|] (11,3.96) -- (12.5,4.89);
	\node at (14.9,9.0) {$\sor(\alpha_s^n(a))$};
	\node at (15.5,8.1) {$\range(\alpha_s^n(a))$};
\end{scope}[style=axes]
\begin{scope}[style=axes]
	\draw[|-] (17,2.5) -- (12.36,10);
	\draw[|-] (17.1789,2.61) -- (12.6,10);
	\draw[-] (12.36,0) -- (11,2.2);
	\draw[-] (12.6,0) -- (11,2.6);
	\draw[-] (21,2.2) -- (16.18,10);
	\draw[-] (21,2.6) -- (16.42,10);
	\draw[-|] (16.18,0) -- (15,1.89);
	\draw[-|] (16.42,0) -- (15.18,2);
	\node at (13.5,2.0) {$\sor(\alpha_u^{-n}(b))$};
	\node at (15.6,2.5) {$\range(\alpha_u^{-n}(b))$};
\end{scope}[style=axes]
\begin{scope}[style=axes]
	\pgfpathcircle{\pgfpoint{18.42cm}{6.38cm}} {1.5pt};
	\pgfusepath{fill}
	\node at (18.0,6.5) {$x_1$};
	\pgfpathcircle{\pgfpoint{18.6cm}{6.49cm}} {1.5pt};
	\pgfusepath{fill}
	\node at (18.7,6.9) {$x_4$};
	\pgfpathcircle{\pgfpoint{18.69cm}{6.35cm}} {1.5pt};
	\pgfusepath{fill}
	\node at (19.1,6.3) {$x_3$};
	\pgfpathcircle{\pgfpoint{18.51cm}{6.24cm}} {1.5pt};
	\pgfusepath{fill}
	\node at (18.4,5.85) {$x_2$};
\end{scope}
\end{tikzpicture}}
\caption{The points $x_1,x_2,x_3,x_4$ for a hyperbolic toral automorphism.}
\label{fig:HTA_alpha_support}
\end{center}
\end{figure}

This completes the interactions of $S(X,\varphi,Q)$ and $U(X,\varphi,P)$ on $\h$. Our goal is to produce an extension of $R^S(X,\varphi,Q) \tensor R^U(X,\varphi,P)$. To accomplish this we shall represent each of these $C^\ast$-algebras as operators on a Hilbert space such that they commute modulo compact operators. Consider the Hilbert space
\[ \ho = \h \tensor \ell^2(\mathbb{Z}) = \bigoplus_{n \in \mathbb{Z}} \h. \]
We shall define representations of $R^S(X,\varphi,Q)$ and $R^U(X,\varphi,P)$ as bounded operators on $\ho$ and show that the interaction of these algebras naturally gives rise to an extension of $R^S(X,\varphi,Q) \tensor R^U(X,\varphi,P)$ by the compact operators of $\ho$.

Recall that for $\delta_x$ in $\h$ we have the unitary operator $u \delta_x = \delta_{\varphi(x)}$ and $\alpha_s(a) = uau^\ast$ and $\alpha_u(b) = ubu^\ast$. The bilateral shift on $\ell^2(\mathbb{Z})$ will be denoted by $B$ and is the operator given by $B \delta_n = \delta_{n-1}$. We note that from this point forwards we will always use $\delta_m$ and $\delta_n$ as basis vectors of $\ell^2(\mathbb{Z})$ and $\delta_x$, $\delta_y$ and $\delta_z$ as basis vectors of $\h = \ell^2(X^h(P,Q))$. Finally, let us also define the operator 
\begin{eqnarray} \label{61}
\mathfrak{U} & = & \bigoplus_{n \in \mathbb{Z}} u^n = \left( \begin{array}{ccccc}
\ddots & & & & \\
 & u^{-1} & & & \\
 & & u^0 & &  \\
 & & & u^1 & \\
 & & & & \ddots \\ \end{array} \right) \in \mathcal{B}(\h \tensor \ell^2(\mathbb{Z})). 
 \end{eqnarray}
 
Define $\overline{\pi}_s: R^S(X,\varphi,Q) \rightarrow \mathcal{B}(\h \tensor \ell^2(\mathbb{Z}))$, for $a$ in $S(X,\varphi,Q)$, via
\[ \overline{\pi}_s(a) = \bigoplus_{n \in \mathbb{Z}} \alpha_s^n(a) = \mathfrak{U}(a \tensor 1)\mathfrak{U}^\ast \qquad \qquad \overline{\pi}_s(u)=1 \tensor B. \]
Also define $\overline{\pi}_u: R^U(X,\varphi,P) \rightarrow \mathcal{B}(\h \tensor \ell^2(\mathbb{Z}))$, for $b$ in $U(X,\varphi,P)$, via
\[ \overline{\pi}_u(b) = b \tensor 1 \qquad \qquad \qquad \overline{\pi}_u(u)=u \tensor B^\ast. \]
The reader is invited to check that these are covariant representations of the Ruelle algebras.

We now consider the interactions of $S(X,\varphi,Q)$, $U(X,\varphi,P)$, $R^S(X,\varphi,Q)$, and $R^U(X,\varphi,P)$ as operators on $\ho$.

\begin{lemma} \label{compact_commutator}
For any  $f$  in $R^S(X,\varphi,Q)$ and  $g$  in $R^U(X,\varphi,P)$, we have 
\[ [\overline{\pi}_s(f),\overline{\pi}_u(g)] = \overline{\pi}_s(f)\overline{\pi}_u(g) - \overline{\pi}_u(g)\overline{\pi}_s(f) \]
is a compact operator on $\ho$.
\end{lemma}

\begin{proof}
From lemma \ref{lemma:SU_cmpct} we know that on each coordinate of $\ho$, for $a$ in $S(X,\varphi,Q)$ and $b$ in $U(X,\varphi,P)$, we have $\overline{\pi}_s(a)\overline{\pi}_u(b)$ and $\overline{\pi}_u(b)\overline{\pi}_s(a)$ are compact operators. Denote the $n^{th}$ coordinate of
\[ \ho = \bigoplus_{n \in \mathbb{Z}} \h \]
by $\ho_n$ and set $\ep >0$. Lemma \ref{SU_eventually_zero} implies that there exists $N_1$ such that for $n \geq N_1$ we have that both $\|\alpha^{-n}(a)b\| < \ep/2$ and $\|b\alpha^{-n}(a)\| < \ep/2$. Therefore,
\begin{eqnarray*}
\|(\overline{\pi}_s(a)\overline{\pi}_u(b) - \overline{\pi}_u(b)\overline{\pi}_s(a))|_{\ho_{-n}} \| & = & \|\alpha^{-n}(a)b - b\alpha^{-n}(a) \| \\
& \leq & \|\alpha^{-n}(a)b\| + \|b\alpha^{-n}(a) \|< \ep.
\end{eqnarray*}
Moreover, Lemma \ref{lemma:asymp_abel} implies that there exists $N_2$ such that for $n \geq N_2$ we have
\[ \|\overline{\pi}_s(a)\overline{\pi}_u(b) - \overline{\pi}_u(b)\overline{\pi}_s(a))|_{\ho_n} \| = \|\alpha^n(a)b - b\alpha^n(a) \| < \ep. \]
Therefore, for $a \in S(X,\varphi,Q)$ and $b \in U(X,\varphi,P)$ we have $[\overline{\pi}_s(a),\overline{\pi}_u(b)]$ is compact. Moreover, computations show that $[\overline{\pi}_s(a), \overline{\pi}_u(u)] = 0$, $[\overline{\pi}_u(b),\overline{\pi}_s(u)]=0$, and $[\overline{\pi}_u(u),\overline{\pi}_s(u)]=0$. The conclusion follows.
\end{proof}

The proof of the next lemma is omitted other than to note that the result follows immediately from the irreducibility of the Smale space itself.

\begin{lemma} \label{lemma:SU_noncmpct}
If $a$ is in $S(X,\varphi,Q)$ and $b$ is in $U(X,\varphi,P)$, then $\overline{\pi}_s(a)\overline{\pi}_u(b)$ and $\overline{\pi}_u(b)\overline{\pi}_s(a)$ are never compact operators on $\ho$ unless either $a$ or $b$ is the zero operator.
\end{lemma}

Define $\mathcal{E}$ to be the $C^\ast$-algebra generated by $\overline{\pi}_s(R^S(X,\varphi,Q))$, $\overline{\pi}_u(R^U(X,\varphi,P))$, and $\mathcal{K}(\ho)$. Note that neither $\overline{\pi}_s(R^S(X,\varphi,Q))$ or $\overline{\pi}_u(R^U(X,\varphi,P))$ contain any compact operators on $\ho$ other than the zero operator. Lemma \ref{compact_commutator} implies that $\overline{\pi}_s(R^S(X,\varphi,Q))$ and $\overline{\pi}_u(R^U(X,\varphi,P))$ commute modulo the compact operators $\mathcal{K}(\ho)$. From this we have that
\[ \mathcal{E} / \mathcal{K}(\ho) \cong R^S(X,\varphi,Q) \tensor R^U(X,\varphi,P). \]
Therefore, we obtain an extension $\Delta$ in $KK^1(R^S(X,\varphi,Q) \tensor R^U(X,\varphi,P),\C)$.

\begin{definition}\label{class_Delta}
The class $\Delta$ in $KK^1(R^S(X,\varphi,Q) \tensor R^U(X,\varphi,P), \cp)$ is represented by the extension
\begin{diagram}
0 & \rTo &\mathcal{K}(\ho) & \rTo & \mathcal{E} & \rTo & R^S(X,\varphi,Q) \tensor R^U(X,\varphi,P) & \rTo & 0. \\
\end{diagram}
\end{definition}

\section{Proof of the main result}

We give a proof of the main result, Theorem \ref{thm:main}. We focus our attention on showing that $\delta \tensor_{R^U(X,\varphi,P)} \Delta =1_{R^S(X,\varphi,Q)}$ and note that an analogous argument shows that $\delta \tensor_{R^S(X,\varphi,Q)} \Delta =-1_{R^U(X,\varphi,P)}$. The proof is divided into roughly 3 parts. In the first we describe the element $\delta \tensor_{R^U(X,\varphi,P)} \Delta$ as an extension. In the second part, we apply a type of untwisting to this extension. Finally, we show that, up to unitary equivalence and Bott periodicity, the class we have obtained is represented by $1_{R^S(X,\varphi,Q)}$.

\begin{lemma} \label{E_prime_extension}
The class of $\delta \tensor_{R^U(X,\varphi,P)} \Delta$ in $KK^1(R^S(X,\varphi,Q) \tensor \mathscr{S},R^S(X,\varphi,Q))$ is given by the extension
\begin{diagram}
0 & \rTo & R^S(X,\varphi,Q) \tensor \mathcal{K}(\ho) & \rTo & \mathcal{E^\prime} & \rTo^{\sigma_* \circ \pi^{\prime} \,} & R^S(X,\varphi,Q) \tensor \mathscr{S} & \rTo & 0. \\
\end{diagram}
\end{lemma}

\begin{proof}
Recall that the expanded product is
\[ \delta \tensor_{R^U(X,\varphi,P)} \Delta = \sigma_*(\tau_{\mathscr{S}}\tau_{R^S(X,\varphi,Q)}(\delta) \tensor \tau^{R^S(X,\varphi,Q)}\sigma_*(\Delta)). \]
The Kasparov product is obtained by composing the $\ast$-homomorphism $\tau_{\mathscr{S}}\tau_{R^S(X,\varphi,Q)}(\delta)$, which is given by
\[ \delta \tensor 1: \mathscr{S} \tensor R^S(X,\varphi,Q) \longrightarrow R^S(X,\varphi,Q) \tensor R^U(X,\varphi,P) \tensor R^S(X,\varphi,Q), \]
and the representation $1 \tensor \overline{\pi}_u \tensor \overline{\pi}_s : R^S(X,\varphi,Q) \tensor R^U(X,\varphi,P) \tensor R^S(X,\varphi,Q) \rightarrow \mathcal{B}(\h \tensor \h \tensor \ell^2(\mathbb{Z}))$ defining the extension $\tau^{R^S(X,\varphi,Q)}\sigma_*(\Delta)$, given by
\[ 0 \rightarrow R^S(X,\varphi,Q) \tensor \mathcal{K}(\ho) \rightarrow R^S(X,\varphi,Q) \tensor \mathcal{E} \rightarrow R^S(X,\varphi,Q) \tensor R^U(X,\varphi,P) \tensor R^S(X,\varphi,Q) \rightarrow 0. \]
Whence, the class $\delta \tensor_{R^U(X,\varphi,P)} \Delta$ is given by the extension
\begin{diagram}
0 & \rTo & R^S(X,\varphi,Q) \tensor \mathcal{K}(\ho) & \rTo & \mathcal{E^\prime} & \rTo & R^S(X,\varphi,Q) \tensor \mathscr{S} & \rTo & 0 \\
\end{diagram}
where $\mathcal{E}^\prime$ is the $C^\ast$-algebra generated by $R^S(X,\varphi,Q) \tensor \mathcal{K}(\ho)$ and the image of the map 
\[ (1 \tensor \overline{\pi}_u \tensor \overline{\pi}_s) \circ  (\delta \tensor 1) \circ \sigma_*:  R^S(X,\varphi,Q) \tensor \mathscr{S} \rightarrow \mathcal{B}(\h \tensor \h \tensor \ell^2(\mathbb{Z})). \]
\end{proof}

We note that the image of $a \cdot u^k \tensor z-1$ completely determines the map
\begin{eqnarray*}
(1 \tensor \overline{\pi}_u \tensor \overline{\pi}_s) \circ  (\delta \tensor 1) \circ \sigma_* & : &  R^S(X,\varphi,Q) \tensor \mathscr{S} \rightarrow \mathcal{B}(\h \tensor \h \tensor \ell^2(\mathbb{Z})).
\end{eqnarray*}
Recall that both $p_G$ and $v$ are elements of $R^S(X,\varphi,Q) \tensor R^U(X,\varphi,P)$ (see Lemma \ref{lemma:p_in_S_U} on page \pageref{lemma:p_in_S_U}). Moreover, let us extend the element $\mathfrak{U}$, defined by (\ref{61}) on page \pageref{61}, to $\h \tensor \h \tensor \ell^2(\mathbb{Z})$ via
\begin{eqnarray*}
\mathfrak{U} & = & \bigoplus_{n \in \mathbb{Z}} u^n \tensor u^n = \left( \begin{array}{ccccc}
\ddots & & & & \\
 & u^{-1} \tensor u^{-1} & & & \\
 & & u^0 \tensor u^0 & &  \\
 & & & u^1\tensor u^1 & \\
 & & & & \ddots \\ \end{array} \right) \in \mathcal{B}(\h \tensor \h \tensor \ell^2(\mathbb{Z})). 
 \end{eqnarray*}
 Now we describe the map $(1 \tensor \overline{\pi}_u \tensor \overline{\pi}_s) \circ  (\delta \tensor 1) \circ \sigma_*:R^S(X,\varphi,Q) \tensor \mathscr{S} \rightarrow \mathcal{B}(\h \tensor \h \tensor \ell^2(\mathbb{Z}))$ on generators:
\begin{eqnarray*}
1 \tensor z & \mapsto & ((u \tensor u)p_G v^\ast) \tensor 1, \\
a \tensor 1 & \mapsto & [((u \tensor u)p_G (u \tensor u)^\ast) \tensor 1)][\mathfrak{U}(1 \tensor a \tensor 1)\mathfrak{U}^\ast], \\
u \tensor 1 & \mapsto & ((u \tensor u)p_G (u \tensor u)^\ast) \tensor B.
\end{eqnarray*}

Now that we have computed the product $\delta \tensor_{R^U(X,\varphi,P)} \Delta$ and established the notation in the sequel, we begin the untwisting step. In particular, we shall define an automorphism $\varTheta: R^S(X,\varphi,Q) \tensor \mathscr{S} \rightarrow R^S(X,\varphi,Q) \tensor \mathscr{S}$ which is homotopic to $1_{R^S(X,\varphi,Q) \tensor \mathscr{S}}$ in $KK(R^S(X,\varphi,Q) \tensor \mathscr{S},R^S(X,\varphi,Q) \tensor \mathscr{S})$ and therefore taking the intersection product of $\delta \tensor_{R^U(X,\varphi,P)} \Delta$ with $\varTheta$ does not change the class. Indeed, for $a \in S(X,\varphi,Q)$, $u$ implementing the action $\alpha_s$, and $f(t) \in \mathscr{S} = C_0(0,1)$, define
\[ \varTheta(a \cdot u^k \tensor f(t)) = a \cdot u^k \tensor e^{2\pi i kt}f(t). \]
We note that this map is given on generators by
\[ 1 \tensor z \mapsto 1 \tensor z \qquad , \qquad a \tensor 1 \mapsto a \tensor 1 \qquad , \qquad u \tensor 1 \mapsto u \tensor z. \]
At this point we need to accomplish two things. First, we must show that $\varTheta$ is homotopic to $1_{R^S(X,\varphi,Q) \tensor \mathscr{S}}$ and second, we must compute the product
\[ \varTheta \tensor_{(R^S(X,\varphi,Q)) \tensor \mathscr{S}} (\delta \tensor_{R^U(X,\varphi,P)} \Delta). \]

\begin{lemma}\label{homotopic_1}
There exists an automorphism $\varTheta: R^S(X,\varphi,Q) \tensor \mathscr{S} \rightarrow R^S(X,\varphi,Q) \tensor \mathscr{S}$ given by
\[ \varTheta(a \cdot u^k \tensor f(t)) = a \cdot u^k \tensor e^{2\pi i kt}f(t). \]
Moreover, $\varTheta$ is the identity in $KK(R^S(X,\varphi,Q) \tensor \mathscr{S}, R^S(X,\varphi,Q) \tensor \mathscr{S})$
\end{lemma}

\begin{proof}
We will show a homotopy from $\varTheta$ to the identity. Indeed, for $r \in [0,1]$ define,
\[ \varTheta_r(a \cdot u^k \tensor f(t)) = a \cdot u^k \tensor e^{2\pi i krt}f(t). \]
We want to show that $\varTheta_r$ is a $\ast$-automorphism for every $r$ in $[0,1]$ and $t$ in $(0,1)$. To see that covariance is maintained, for all $r$ in $[0,1]$ and $t$ in $(0,1)$, we compute
\begin{eqnarray*}
\varTheta_r(u \tensor 1) \varTheta_r(a \tensor f(t)) \varTheta_r(u \tensor 1)^\ast & = & (u \tensor e^{2 \pi i rt}) (a \tensor f(t)) (u^\ast \tensor e^{-2 \pi i rt}) \\
& = & u a u^\ast \tensor f(t) = \alpha_s(a) \tensor f(t) = \varTheta_r(\alpha_s(a) \tensor f(t)).
\end{eqnarray*}
Therefore, the map $\varTheta_r$ satisfies the covariance conditions for all $r \in [0,1]$ and $t$ in $(0,1)$, so extends to a $\ast$-homomorphism on $R^S(X,\varphi,Q) \tensor \mathscr{S}$. We can explicitly write a formula for the inverse of $\varTheta_r$ on generators so that $\varTheta_r$ is actually a $\ast$-automorphism. Moreover, $\varTheta_r$ is clearly is faithful since $R^S(X,\varphi,Q) \tensor \mathscr{S}$ is simple and $\varTheta$ is clearly onto. Now we must show that each map
\[ a \cdot u^k \tensor f(t) \mapsto \varTheta_r(a \cdot u^k \tensor f(t)) =  a \cdot u^k \tensor e^{2\pi i k t r}f(t)\]
is continuous. Let $\ep >0$, and set $\delta = \frac{\ep}{k M}$ where $M$ is the maximum value of $\|a \cdot u^k \tensor f(t)\|$ for $t \in (0,1)$. For $|r - r^\prime|<\delta$ we compute
\begin{eqnarray*}
\| \varTheta_r(a \cdot u^k \tensor f(t)) - \varTheta_{r^\prime}(a \cdot u^k \tensor f(t)) \| & = & \| a \cdot u^k \tensor e^{2\pi i k t r} f(t) - a \cdot u^k \tensor e^{2\pi i k t r^\prime} f(t) \| \\
& = & \| a \cdot u^k \tensor (1 - e^{2\pi i k t (r^\prime - r)}) f(t) \| \\
& = & |1 - e^{2\pi i k t (r^\prime - r)}|\| a \cdot u^k \tensor f(t) \| \\
& < & \frac{\ep}{M}\| a \cdot u^k \tensor f(t) \| \leq \ep.
\end{eqnarray*}
Finally,
\begin{eqnarray*}
\varTheta_0(a \cdot u^k \tensor f(t)) & = & a \cdot u^k \tensor e^{2\pi i k 0t}f(t)= a \cdot u^k \tensor f(t) \\
\varTheta_1(a \cdot u^k \tensor f(t)) & = & a \cdot u^k \tensor e^{2\pi i k 1t}f(t) = \varTheta(a \cdot u^k \tensor f(t))
\end{eqnarray*}
so that $\varTheta_0 = \textrm{Id}$ and $\varTheta_1 = \varTheta$. Therefore, $\varTheta$ is homotopic to $1_{R^S(X,\varphi,Q) \tensor \mathscr{S}}$ in $KK(R^S(X,\varphi,Q) \tensor \mathscr{S},R^S(X,\varphi,Q) \tensor \mathscr{S})$.
\end{proof}

\begin{lemma}\label{E_prime_prime_extension}
The class of $\varTheta \tensor_{R^S(X,\varphi,Q) \tensor \mathscr{S}} (\delta \tensor_{R^U(X,\varphi,P)} \Delta)$ in \\ $KK^1(R^S(X,\varphi,Q) \tensor \mathscr{S},R^S(X,\varphi,Q))$ is represented by the extension
\begin{diagram}
0 & \rTo & R^S(X,\varphi,Q) \tensor \mathcal{K}(\ho) & \rTo & \mathcal{E^{\prime \prime}} & \rTo & R^S(X,\varphi,Q) \tensor \mathscr{S} & \rTo & 0. \\
\end{diagram}
Furthermore, this extension represents the same class in $KK$-theory as $\delta \tensor_{R^U(X,\varphi,P)} \Delta$.
\end{lemma}

The proof is analogous to the proof of Lemma \ref{E_prime_extension}. However, we note that $\mathcal{E}^{\prime\prime}$ is the $C^\ast$-algebra generated by $R^S(X,\varphi,Q) \tensor \mathcal{K}(\ho)$ and the image of the map 
\[ (1 \tensor \overline{\pi}_u \tensor \overline{\pi}_s) \circ  (\delta \tensor 1) \circ \sigma_* \circ \varTheta:  R^S(X,\varphi,Q) \tensor \mathscr{S} \rightarrow \mathcal{B}(\h \tensor \h \tensor \ell^2(\mathbb{Z})). \]
Furthermore, the image of $a \cdot u^k \tensor z-1$ completely determines the above map, which is described on generators by
\begin{eqnarray}
\label{74}
1 \tensor z & \mapsto & ((u \tensor u)p_G v^\ast) \tensor 1, \\
\label{75}
a \tensor 1 & \mapsto & [((u \tensor u)p_G (u \tensor u)^\ast) \tensor 1)][\mathfrak{U}(1 \tensor a \tensor 1)\mathfrak{U}^\ast], \\
\label{76}
u \tensor 1 & \mapsto & ((u \tensor u)p_G v^\ast) \tensor 1.
\end{eqnarray}

To complete the proof we must show that the class
\[ \varTheta \tensor_{R^S(X,\varphi,Q) \tensor \mathscr{S}} (\delta \tensor_{R^U(X,\varphi,P)} \Delta) \]
in $KK^1(R^S(X,\varphi,Q) \tensor \mathscr{S},R^S(X,\varphi,Q))$ is equivalent to $\tau^{R^S(X,\varphi,Q)}(\mathcal{T}_0)$. Once we have accomplished this then Bott periodicity implies the result, see Section \ref{sec:dual}.

We begin with a technical construction to produce a unitary operator. Suppose that $(\mathcal{F},G)$ is an $\epsilon^\prime_{X}$-partition of $X$, as in section \ref{class_delta}. We define a vector, which we denote $\chi_{G}$ in $\h=\ell^2(X^h(P,Q))$ which takes the value $(\# G)^{-1/2}$ on the set $G$ and zero elsewhere. Note that $\chi_{G}$ is a unit vector. We let $q_{G}$ denote the rank one projection onto the span of $\chi_{G}$.

We define $W_G$ by setting, for all $y$ in $X^h(P,Q)$,
\[ W_G (\delta_{y} \tensor \chi_{G}) = \sum_{k} f_{k}(y) \delta_{[y, g_{k}]} \tensor \delta_{[g_{k}, y]}, \]
where the sum is taken over all $k$ such that $y$ is in $B(g_{k}, \ep'_{X}/2)$. Recall that in an $\ep'_{X}$-partition of $X$ the support of $f_{k}$ is contained in $B(g_{k}, \ep'_{X}/2)$. Using our standard convention (the bracket returns zero if it is not defined), we will simply write the sum above as being over all $k=1, 2, \ldots, K$ since $f_{k}(y)$ will be zero if $[y, g_{k}]$ and $[g_{k}, y]$ fail to be defined. We also set $W_G (\delta_{z} \tensor \xi) = 0$, 
for $ \xi$ in $\h$ orthogonal to $\chi_{G}$. 

It is easy to verify that 
\[ W_G^{*} (\delta_{y} \tensor \delta_{z}) 
= \left\{ \begin{array}{ll} 
f_{k}([y, z]) \delta_{[y,z]} \tensor \chi_{G} & \text{ if } y \in X^u(g_{k}, \ep), z \in X^s(g_{k}, \ep)  \\ 
0  &  \text{ otherwise } \end{array} \right. \]
for all $w, z$ in $X^h(P,Q)$. The following lemma summarizes the basic properties of $W_G$.

\begin{lemma}\label{lemma:W}
Suppose that $(\mathcal{F},G)$ is an $\epsilon^\prime_{X}$-partition of $X$ and $W_G$ is defined as above. Then
\begin{enumerate}
\item $W_G^\ast W_G = 1 \tensor q_{G}$.
\item $W_G W_G^\ast = p_G$.
\item If $(\mathcal{F} \circ \varphi^{-1}, \varphi(G))$ is also an $\ep'_{X}$-partition, then 
\[ (u \tensor u)W_G(u \tensor u)^{*} = W_{\varphi(G)}. \]
\item $W_G (S(X,\varphi,Q) \tensor \mathcal{K}) \subset S(X,\varphi,Q) \tensor \mathcal{K}$.
\end{enumerate}
\end{lemma}

\begin{proof}
The first three items are the result of direct computations, which we omit. For the fourth item, let $a$ be in $S(X,\varphi,Q)$ and suppose $k$ is any compact operator. Now
\[ W_G(a \tensor k) = W_G(1 \tensor q_{G})(a \tensor k) = W_G (a \tensor q_{G}) (1 \tensor k), \]
and so it suffices to show $W_G(a \tensor q_{G})$ is in $S(X,\varphi,Q) \tensor \mathcal{K}$. The $C^\ast$-algebra $S(X,\varphi,Q)$ has an approximate identity consisting of continuous functions of compact support on $X^u(Q)$. Moreover, such functions are spanned by elements supported on sets of the form $X^u(v,\epsilon'_{X}/2)$. So it suffices to consider a point $v$ in $X^u(Q)$, a function $a$ in $S(X,\varphi,Q)$ supported on a basic set of the form $V^s(v,v,h^s,\epsilon'_{X}/2)$ such that
$a \delta_{y} = a(y,y)\delta_{y}$ if $y$ is in $X^u(v,\epsilon'_{X}/2)$ and zero otherwise, and prove that $W_G(a \tensor q_{G})$ is in $S(X,\varphi,Q) \tensor \mathcal{K}$.

For each $k$, define a function $b_{k}$ supported on a basic set of the form $V^s([v,g_k],v,h^s,\epsilon'_X /2)$ by $b_k(y', y) = a(y,y)f_{k}(y)$ if $d(y, g_{k}) < \epsilon'_{X}$ and $[y, g_{k}] = y'$ and to be zero otherwise. Also define $e_{k}$ to be the rank one operator which maps $\chi_{G}$ to $\delta_{[g_k, v]}$ and is zero on the orthogonal complement of $\chi_{G}$. It follows that $b_k$ is in $S(X,\varphi,Q)$and a computation shows that
\begin{eqnarray*} 
W_G (a \tensor q_{G}) & = & \sum_k b_k \tensor e_{k} \in S(X,\varphi,Q) \tensor \mathcal{K}. 
\end{eqnarray*}
\end{proof}

\begin{lemma} \label{lemma:key}
Let $a$ be in $S(X, \varphi, Q)$ and let $(\mathcal{F}, G)$ be an 
$\ep'_{X}$-partition. Then we have 
\[ \lim_{n \rightarrow \infty} \| (1 \tensor \alpha_s^n(a))W_G - W_G(\alpha_s^n(a)  \tensor 1) \| = 0. \]
\end{lemma}

\begin{proof}
It suffices to prove the result for $a$ supported in a basic set of the form $V^s(v,w,h^s,\delta)$ and further, since we are taking limits as $n$ goes to positive infinity, we may also assume that $v$ and $w$ are within $\ep'_X /2$ so that $h^s$ is given by the bracket map.

We observe that both operators $(1 \tensor \alpha_s^n(a)) W_G$ and $W_G (1 \tensor \alpha_s^n(a)) $ are zero on the orthogonal complement of $\ell^2(X^h(P,Q)) \tensor \cp \cdot \chi_{G}$. We consider $y$ in $\h = \ell^2(X^h(P,Q))$ and compute
\begin{eqnarray*}
(1 \tensor \alpha_s^n(a)) W_G( \delta_y \tensor \chi_{G}) & = & (1 \tensor \alpha_s^n(a)) \sum_{k} f_{k}(y) \delta_{[y, g_{k}]} \tensor \delta_{[g_{k}, y]} \\
& = & \sum_{k} f_{k}(y) a([\varphi^{-n}[g_{k}, y], v],\varphi^{-n}[g_{k}, y]) \delta_{[y, g_{k}]} \tensor \delta_{\varphi^{n}[\varphi^{-n}[g_{k}, y], v]} 
\end{eqnarray*}
and also
\begin{eqnarray*}
&& W_G (\alpha_s^n(a) \tensor 1) ( \delta_{y} \tensor \chi_{G}) = W_G a([\varphi^{-n}(y), v],\varphi^{-n}(y)) \delta_{\varphi^{n}[\varphi^{-n}(y), v]} \tensor \chi_{P} \\
&&  \qquad = \sum_{k} f_{k}(\varphi^{n}[\varphi^{-n}(y), v]) a([\varphi^{-n}(y), v],\varphi^{-n}(y)) \delta_{[\varphi^{n}[\varphi^{-n}(y), v], g_{k}]} \tensor \delta_{[g_{k}, \varphi^{n}[\varphi^{-n}(y), v]]}. 
\end{eqnarray*}

Let $\ep > 0$ be given. Let $M$ be an upper bound on the function $|a|$. We may find a constant $\ep_{1} > 0$ such that $| f_{k}(y) - f_{k}(z) | < \epsilon/2MK$, for all $y, z$ with $d(y, z)  < \epsilon_{1}$ and $1 \leq k \leq K$. In addition, we select $\epsilon_{1} > 0$ such that $|a([y,v],y) - a([z,v],z)| < \epsilon/2K$ for all $y, z$ with $z$ in $X^{u}(y, \epsilon_{1})$. We choose $N$ sufficiently large so that $\lambda^{-n} \epsilon_{X}/2 < \epsilon_{1}$ and $\lambda^{-n} \epsilon_{X} < \epsilon'_{X}$, for all $n \geq N$. With $n \geq N$, holding $k$ fixed for the moment, we make the claim that if the coefficient in either expression above:
\[ f_{k}(y) a([\varphi^{-n}[g_{k}, y], v],\varphi^{-n}[g_{k}, y]) \qquad \textrm{ or } \qquad f_{k}(\varphi^{n}[\varphi^{-n}(y), v]) a([\varphi^{-n}(y), v],\varphi^{-n}(y)) \]
is not zero, then we have 
\begin{enumerate}
 \item $[y, g_{k}] = [\varphi^{n}[\varphi^{-n}(y), v], g_{k}]$,
\item $ \varphi^{n}[\varphi^{-n}[g_{k},y], v] = [g_{k}, \varphi^{n}[\varphi^{-n}(y), v]]$,
\item the map sending $(y, g_{k})$ to $([y, g_{k}],  \varphi^{n}[\varphi^{-n}[g_{k},y], v] )$ is injective,
\item $\varphi^{-n}[g_{k}, y]$ is in $X^{u}(\varphi^{-n}(y), \epsilon_{1})$,
\item $d(\varphi^{n}[\varphi^{-n}(y), v], y) < \epsilon_{1}$.
\end{enumerate}

If $f_{k}(y) a([\varphi^{-n}[g_{k}, y], v],\varphi^{-n}[g_{k}, y])$  is non-zero, then $f_{k}(y)$ must be non-zero and this means that $y$ is in $B(g_{k}, \epsilon'_{X}/2)$. Moreover, from the choice of $\epsilon'_{X}$, we have that $[g_{k}, y]$ is in $X^{u}(y, \epsilon_{X}/2)$ and hence $\varphi^{-n}([g_{k}, y])$ is in $X^{u}(\varphi^{-n}(y) \lambda^{-n} \epsilon_{X}/2)$. In addition, we know that $a([\varphi^{-n}[g_{k}, y], v],\varphi^{-n}[g_{k}, y])$ is non-zero and this means that $\varphi^{-n}([g_{k}, y])$ is in $X^{u}(w, \epsilon'_{X}/2)$ and it follows that $\varphi^{-n}(y)$ is in $X^{u}(w, \epsilon'_{X}/2 + \lambda^{-n} \epsilon_{X}/2)$. Since $\lambda^{-n} \epsilon_{X} < \epsilon'_{X} < \epsilon_{X}/2$, we see that $[\varphi^{-n}(y), v]$ is also defined and is in $X^{s}(\varphi^{-n}(y), \epsilon_{X}/2)$. It follows that $\varphi^{n}[\varphi^{-n}(y), v]$ is in $X^{s}(y, \lambda^{-n}\epsilon_{X}/2 )$ and also in $B(g_{k}, \epsilon'_{X})$.

If the second expression, $f_{k}(\varphi^{n}[\varphi^{-n}(y), v]) a([\varphi^{-n}(y), v],\varphi^{-n}(y))$ is non-zero, then we must have that $a([\varphi^{-n}(y), v],\varphi^{-n}(y))$ is non-zero and so $\varphi^{-n}(y)$ is in $X^{u}(w, \epsilon'_{X}/2)$. Then we have $[\varphi^{-n}(y), v]$ is in $X^{s}(\varphi^{-n}(y), \epsilon_{X}/2)$ and hence
$\varphi^{n}[\varphi^{-n}(y), v]$ is in $X^{s}(y, \lambda^{-n} \epsilon_{X}/2)$. In addition, if the coefficient is non-zero, the $f_{k}$ term is non-zero and this means that this same point is in $B(g_{k}, \epsilon'_{X}/2)$ and hence, $y$ is in $B(g_{k}, \epsilon'_{X}/2 + \lambda^{-n} \epsilon_{X}/2)$. Since $\lambda^{-n} \epsilon_{X} < \epsilon'_{X}$, $[g_{k}, y]$ is in $X^{u}(y, \epsilon_{X}/2)$. 

To summarize, if either term is non-zero, then we have $[g_{k}, y]$ is defined and is in $X^{u}(y, \epsilon_{X}/2)$, $\varphi^{-n}(y)$ is in 
$X^{u}(w, \epsilon'_{X})$ and $[\varphi^{-n}(y), v]$ is defined and in $X^{s}(\varphi^{-n}(y), \epsilon_{X}/2)$. Parts $4$ and $5$ of the claim follow at once since $\lambda^{-n} \epsilon_{X}/2 < \epsilon_{1}$.

For any $0 \leq m \leq n$, we have 
\[ d(\varphi^{m-n}(y), \varphi^{m}[\varphi^{-n}(y), x]) \leq \lambda^{-m} d(\varphi^{-n}(y), [\varphi^{-n}(y), v]) \leq \lambda^{-m} \epsilon_{X}/2 \leq \epsilon_{X}/2, \]
and 
\[ d(\varphi^{m-n}(y), \varphi^{m-n}[g_{k}, y]) \leq \lambda^{-n+m} d(y, [g_{k}, y]) \leq  \lambda^{-n+m} \epsilon_{X}/2 \leq  \epsilon_{X}/2. \]
From the triangle inequality, we have 
\[ d( \varphi^{m}[\varphi^{-n}(y), v], \varphi^{m-n}[g_{k}, y]) \leq \epsilon_{X}. \]
This means that the bracket of these points is defined (in either order). First, taking bracket in the order given and using the $\varphi$-invariance of the bracket we have
\begin{eqnarray*}
[\varphi^{m}[\varphi^{-n}(y), v], \varphi^{m-n}[g_{k}, y]]  &  =  &  \varphi^{m}[ [\varphi^{-n}(y), v], \varphi^{-n}[g_{k}, y] ]  \\
& = & \varphi^{m}[ \varphi^{-n}(y), \varphi^{-n}[g_{k}, y] ].
\end{eqnarray*}
When $m=n$, the left hand side becomes
\[ [\varphi^{n}[\varphi^{-n}(y), v], \varphi^{n-n}[g_{k}, y]] = [  \varphi^{n}[\varphi^{-n}(y), v], [g_{k}, y]] = [  \varphi^{n}[\varphi^{-n}(y), v], y]] \]
while the right hand side is
\[ \varphi^{n}[ \varphi^{-n}(y), \varphi^{-n}[g_{k}, y] ] = \varphi^{n}(\varphi^{-n}(y)) = y \]
as $\varphi^{-n}[g_{k}, y]$ is in $X^{u}(\varphi^{-n}(y), \lambda^{-n} \epsilon_{X}/2)$. Now bracketing each with $g_{k}$ yields
\[ [y, g_{k}] = [ [  \varphi^{n}[\varphi^{-n}(y), v], y] , g_{k}] =  [ \varphi^{n}[\varphi^{-n}(y), v], g_{k}] \]
and we have established part 1 of the claim.

On the other hand, if we bracket in the other order, and again use the $\varphi$-invariance, we obtain
\[ [\varphi^{m-n}[g_{k}, y], \varphi^{m}[\varphi^{-n}(y), v]] = \varphi^{m}[\varphi^{-n}[g_{k},y], [\varphi^{-n}(y), v]] = \varphi^{m}[\varphi^{-n}[g_{k},y], v]. \]
Setting $m=n$, the left hand side is 
\[ [\varphi^{n-n}[g_{k}, y], \varphi^{n}[\varphi^{-n}(y), v] ] = [ [g_{k}, y], \varphi^{n}[\varphi^{-n}(y), v] ] = [ g_{k}, \varphi^{n}[\varphi^{-n}(y), v] \]
while the right is 
\[ \varphi^{n}[\varphi^{-n}[g_{k}, y], \varphi^{-n}(y) ]. \]
We have established the second part of the claim.

For the third part of the claim, let $x = [y, g_{k}]$ and $z =[g_{k}, \varphi^{n}[\varphi^{-n}(y), v]]$. We can recover $y$ and $g_{k}$ from $x$ and $z$ by observing that $[z, x] = g_{k}$ and  
\[ [x, z] = [[y, g_{k}], [g_{k}, \varphi^{n}[\varphi^{-n}(y), v]]] = [y, \varphi^{n}[\varphi^{-n}(y), v]] = y, \]
since $\varphi^{n}[\varphi^{-n}(y), v]$ is in $X^{s}(y, \epsilon)$.

We may conclude from the first two parts of our claim that
\begin{eqnarray*}
&& (1 \tensor \alpha_s^n(a)) W_G - W_G (\alpha_s^n(a) \tensor 1 ) ( \delta_{y} \tensor \chi_{G})  \\
&& \quad  = \sum_{k} (f_{k}(y) a([\varphi^{-n}[g_{k}, y], v],\varphi^{-n}[g_{k}, y]) - f_{k}(\varphi^{n}[\varphi^{-n}(y), v]) a([\varphi^{-n}(y), v],\varphi^{-n}(y))) \\ 
&& \hspace{10cm} \delta_{[y, g_{k}]} \tensor \delta_{[g_{k}, \varphi^{n}[\varphi^{-n}(y), v]]}.
\end{eqnarray*}

From part $3$, we see that the vectors appearing in the right hand side of the expression at the end of the last paragraph are pairwise orthogonal and the sums obtained for different values of $y$ are pairwise orthogonal. From this it follows that 
\begin{eqnarray*}
&& \| (1 \tensor \alpha_s^n(a)) W_G - W_G (\alpha_s^n(a) \tensor 1 ) \|^{2} \\
&& \qquad \qquad \qquad =  \sup_{y \in X^h(P,Q)} \sum_{k} | f_{k}(y) a([\varphi^{-n}[g_{k}, y], v],\varphi^{-n}[g_{k}, y])  \\
&& \qquad \qquad \qquad \qquad - f_{k}(\varphi^{n}[\varphi^{-n}(y), v]) a([\varphi^{-n}(y), v],\varphi^{-n}(y)) |^{2}.
\end{eqnarray*}
The estimate that, for fixed $k$ and $y$, 
\[ |f_{k}(y) a([\varphi^{-n}[g_{k}, y], v],\varphi^{-n}[g_{k}, y]) - f_{k}(\varphi^{n}[\varphi^{-n}(y), v]) a([\varphi^{-n}(y), v],\varphi^{-n}(y)) | < \epsilon/K \]
follows from the last two parts of our claim and standard techniques. This completes the proof.
\end{proof}

We next define
\[ \widehat{W_G} = \left( \begin{array}{cc} 
W_G & (1 - W_G W_G^{*})^{1/2} \\
-(1 -  W_G^{*} W_G)^{1/2} & W_G^{*} 
\end{array} \right) \]
which is a unitary operator. Moreover, it follows from part 4 of Lemma \ref{lemma:W} that it is in the multiplier algebra of $S(X, \varphi, Q) \tensor \mathcal{K}$.

\begin{lemma} \label{lemma:key2}
Let $a$ be in $S(X, \varphi, Q)$ and let $(\mathcal{F},G)$ be an 
$\ep$-partition. Then we have 
\[ \lim_{n \rightarrow + \infty} \| ((p_G (1 \tensor \alpha^n(a)) ) \tensor e_{1,1}) \widehat{W_G} - \widehat{W_G}(( \alpha^n(a) \tensor q_{G}) \tensor e_{1,1}) \| = 0. \]
Moreover, if $(\mathcal{F}\circ \varphi^{-1}, \varphi(G))$ is also an $\ep'_X$-partition, then $\widehat{W_{\varphi(G)}}$ is also in the multiplier algebra of $S(X, \varphi, Q) \tensor \mathcal{K}$ and the analogous result holds with $\widehat{W_{\varphi(G)}}$, $p_{\varphi(G)}$, and $q_{\varphi(G)}$.
\end{lemma}

\begin{proof}
First, notice that $p_G$ is in $S \tensor U$, while $\alpha^{n}(a)$ is in $S$. It follows from Lemma \ref{lemma:asymp_abel} that in the first term we have
\[ \lim_{n \rightarrow \infty} \| p_G(1 \tensor \alpha_s^n(a)) - (1 \tensor \alpha_s^n(a))p_G \| = 0. \]
So it suffices to prove the result after interchanging the order of $p_G$ and $1 \tensor \alpha_s^n(a)$.

It follows from the fact that $W_G W_G^{*} = p_G$ that the (new) first term above is
\begin{eqnarray*}
(((1 \tensor \alpha_s^n(a))p_G) \tensor e_{1,1}) \widehat{W_G} & = & ( (1 \tensor \alpha_s^n(a)) \tensor e_{1,1})(p_G \tensor e_{1,1}) \widehat{W_G} \\
& = & ((1 \tensor \alpha_s^n(a)) \tensor e_{1,1}) (W_G \tensor e_{1,1} ) \\
& = & (1 \tensor \alpha_s^n(a))W_G \tensor e_{1,1}.
\end{eqnarray*}
On the other hand, using the fact that $W_G^{*}W_G = 1 \tensor q_{G}$, the second term above is
\begin{eqnarray*}
\widehat{W_G}((\alpha_s^n(a) \tensor q_{G}) \tensor e_{1,1}) & = & \widehat{W_G} ((1 \tensor q_{G}) \tensor e_{1,1})(\alpha_s^n(a) \tensor 1) \tensor e_{1,1} ) \\
& = & (W_G \tensor e_{1,1}) ((\alpha_s^n(a) \tensor 1) \tensor e_{1,1} ) \\
& = &  W_G ((\alpha_s^n(a) \tensor 1) \tensor e_{1,1}).
\end{eqnarray*}
The first statement now follows at once from Lemma \ref{lemma:key} and the second statement follows from combining the first statement with part $3$ of lemma \ref{lemma:W}.
\end{proof}

Let us denote the map 
\[ (1 \tensor \overline{\pi}_u \tensor \overline{\pi}_s) \circ  (\delta \tensor 1) \circ \sigma_* \circ \varTheta:  R^S(X,\varphi,Q) \tensor \mathscr{S} \rightarrow \mathcal{B}(\h \tensor \h \tensor \ell^2(\mathbb{Z})) \]
by $\psi$. Observe that $\psi(a \cdot u^k \tensor z^l - 1)$ determines the extension
\begin{diagram}
0 & \rTo & R^S(X,\varphi,Q) \tensor \mathcal{K}(\ho) & \rTo & \mathcal{E^{\prime \prime}} & \rTo & R^S(X,\varphi,Q) \tensor \mathscr{S} & \rTo & 0,
\end{diagram}
representing the class $\varTheta \tensor_{R^S(X,\varphi,Q) \tensor \mathscr{S}} (\delta \tensor_{R^U(X,\varphi,P)} \Delta)$ in $KK^1(R^S(X,\varphi,Q) \tensor \mathscr{S},R^S(X,\varphi,Q))$. Therefore, using Lemma \ref{lemma:key2} and equivalence in $KK$-theory, we have
\begin{eqnarray*}
\psi(a \cdot u^k \tensor z^l - 1) & = & (\widehat{W_{\varphi(G)}} \tensor 1)^\ast (\psi(a \cdot u^k \tensor z^l - 1) \tensor e_{1,1})(\widehat{W_{\varphi(G)}} \tensor 1) \\
& = & (W_{\varphi(G)} \tensor 1)^\ast (\psi(a \cdot u^k \tensor z^l - 1))(W_{\varphi(G)} \tensor 1).
\end{eqnarray*}
Now, combining Lemma \ref{lemma:SU_cmpct}, Lemma \ref{SU_eventually_zero}, and Lemma \ref{lemma:key2}, we compute
\begin{eqnarray*}
(W_{\varphi(G)}^* \tensor 1)\mathfrak{U}(1 \tensor a \tensor 1)\mathfrak{U}^\ast(W_{\varphi(G)} \tensor 1) & = & \left\{ \begin{array}{cc} \mathfrak{U}(a \tensor 1 \tensor 1)\mathfrak{U}^\ast(1 \tensor q_G \tensor 1) & \textrm{ if } n \geq 0 \\ 0 & \textrm{ if } n < 0 \\ \end{array} \right.
\end{eqnarray*}
where equality is up to the ideal $R^S(X,\varphi,Q) \tensor \mathcal{K}(\ho)$. To further simplify notation, let us also define
\[ U = W_{\varphi(G)}^\ast ((u \tensor u)p_G v^\ast)W_{\varphi(G)}  \in \mathcal{B}(\h \tensor \h).\]

Using the notation and computations from the preceding paragraph together with the maps described in (\ref{74})-(\ref{76}) on page \pageref{74}, we obtain, for $a$ in $S(X,\varphi,Q)$,
\[ \psi(a \cdot u^k \tensor z^l - 1) = (U \tensor B^\ast)^{l+k}\mathfrak{U}(a \tensor 1 \tensor 1)\mathfrak{U}^\ast(1 \tensor B)^k - (U \tensor B^\ast)^k \mathfrak{U}(a \tensor 1 \tensor 1)\mathfrak{U}^\ast(1 \tensor B)^k \]
as a bounded operator on the Hilbert space $\h \tensor q_G \tensor \ell^2(\mathbb{N})$,  where we have replaced $\ell^2(\mathbb{Z})$ by $\ell^2(\mathbb{N})$ since the operator $\psi(a \cdot u^k \tensor z^l - 1)$ is zero on the subspace $\{ 1 \tensor q_G \tensor \delta_n \mid n < 0\}$. Therefore, note that $B$ is a one sided shift on $\ell^2(\mathbb{N})$.

We are left to show that we have the following isomorphism of extensions, where $\mathcal{E}^{\prime\prime\prime}$ is the $C^\ast$-algebra generated by the image of $\psi$ and the ideal $R^S(X,\varphi,Q) \tensor \mathcal{K}(\h \tensor \ell^2(\mathbb{N}))$,
\[ \scalebox{0.87}{
\begin{diagram}
0 & \rTo & R^S(X,\varphi,Q) \tensor q_G \tensor \mathcal{K}(\ell^2(\mathbb{N})) & \rTo & \mathcal{E^{\prime \prime\prime}} & \rTo^{\pi \qquad} & R^S(X,\varphi,Q) \tensor \mathscr{S} & \rTo & 0 \\
 & & \dTo^{\cong} & & \dTo^{\beta} & & \dEq & \\
0 & \rTo & R^S(X,\varphi,Q) \tensor \mathcal{K}(\ell^2(\mathbb{N})) & \rTo & R^S(X,\varphi,Q) \tensor C^\ast(B-1) & \rTo & R^S(X,\varphi,Q) \tensor \mathscr{S} & \rTo & 0. \\
\end{diagram}} \]
Indeed, the quotient map $\pi: \mathcal{E}^{\prime\prime\prime} \rightarrow R^S(X,\varphi,Q) \tensor \mathscr{S}$ is given on generators by
\[ U \tensor 1 \mapsto  u \tensor 1 \quad , \quad 1 \tensor B  \mapsto  u \tensor \overline{z} \quad , \quad \mathfrak{U}(a \tensor 1 \tensor 1)\mathfrak{U}^\ast \mapsto  a \tensor 1 \]
and the map $\beta$ is given on generators by
\[ U \tensor 1 \mapsto  u \tensor 1 \quad , \quad 1 \tensor B  \mapsto  u \tensor B \quad , \quad \mathfrak{U}(a \tensor 1 \tensor 1)\mathfrak{U}^\ast \mapsto  a \tensor 1. \]
The reader is left to show that $\beta$ is an isomorphism and that the above diagram commutes. The second extension represents the class $\tau^{R^S(X,\varphi,Q)}(\mathcal{T}_0)$, which is $KK$-equivalent to $1_{R^S(X,\varphi,Q)}$ by Bott periodicity. This completes the proof.

\section{Concluding remarks and questions}

\subsection{Existence of the duality classes}

As we saw in Section \ref{sec:delta}, the construction of the K-theory duality element did not require the expanding and contracting nature of the dynamics. An essential property was that the stable and unstable relations intersected at a countable set of points.  Recall that a transversal to a foliation is a set which meets each leaf in a countable set, so the condition that each stable equivalence class meets each unstable class in a countable set is a transversality condition.  Based on the example of transverse foliations, a notion of transverse groupoids has been suggested and it is hoped that the axioms will be sufficient for the construction of a K-theory duality class in $KK^i(C^*(\mathcal G_1), C^*(\mathcal G_2))$ when $\mathcal G_1$ and $\mathcal G_2$ are transverse groupoids.

On the other hand, the existence of the class $\Delta$ requires hyperbolicity.  This is what might be expected from the Dirac-dual Dirac approach to the Baum-Connes conjecture.  In that case, the construction of the dual Dirac element, which is the analog of our $\Delta$, requires the use of a non-positively curved space.  For such a space the geodesic flow will be hyperbolic.  It would be interesting to find fundamentally different types of conditions which would lead to a K-homology duality class, or to understand why there no other possibilities.

\subsection{Crossed products and dynamics}

The relation between a hyperbolic group acting (amenably) on its Gromov boundary and hyperbolic dynamical systems is an interesting subject.  One might hope that in general one could do as Bowen and Series did, and  recode the action of certain Fuchsian groups on their boundaries to obtain a single transformation with an associated Markov partition, hence one gets a subshift of finite type.  Spielberg showed that the Ruelle algebras for the subshift are isomorphic to crossed product algebras.  This was extended by Laca-Spielberg and Delaroche, but in those cases it was necessary to use the fact that the algebras satisfied the hypothesis of the Phillips-Kirchberg theorem, and show that their K-theories were the same, to deduce that they were isomorphic.  This suggest that the appropriate setting for relating amenable actions of hyperbolic groups on their boundaries to hyperbolic dynamics is via identifying the crossed product \csa s with Ruelle algebras of Smale spaces.

\subsection{Ruelle algebras as building blocks for algebras associated to diffeomorphisms of manifolds}
 
Let $f:M\to M$ be an Axiom A diffeomorphism of a compact manifold, as defined by Smale \cite{Sma}.  Recall that $M$ can be expressed as a union of submanifolds, $M \supset M_1 \supset \ldots \supset M_n$ where each $M_i \setminus M_{i+1}$ contains a basic set, $S_i$, such that $(f|S_i, S_i)$ is an irreducible Smale space.  Thus, for each $i$, we have a Ruelle algebra, $R^u_i$.  These algebras are determined by their K-theory, analogously to how spheres are determined by their homology groups.  It would be interesting to find a natural algebra associated to the diffeomorphism $f$ which could be constructed from these Ruelle algebras with additional algebraic data.  One possibility is to build the algebra as iterated extensions, with elements of Kasparov groups playing the role of $k$-invariants from topology.   


\end{document}